\documentclass{aptpub}
\authornames{L. Coutin, L. Huang, M. Pontier} 
\shorttitle{Weak Uniqueness for the governing equation of the joint law} 

\usepackage{comment}
\usepackage{color}

\def\blue{\textcolor{blue}}
\def\R{\mathbb R}
\def\N{\mathbb N}
\def\E{\mathbb E}

\def\1{\textbf{1}}

\def\PP{\mathbb P}

\begin{document}

\title{Weak Uniqueness for the PDE Governing the Joint Law of a Diffusion and Its Running Supremum}

\authorone[Institut Mathématiques de Toulouse]{Laure Coutin} 
\addressone{Institut de Mathématiques de Toulouse. CNRS UMR 5219. Université Paul Sabatier, 118 route de Narbonne, F-31062 Toulouse cedex 09.} 
\emailone{coutin@math.univ-toulouse.fr} 

\authortwo[Institut Mathématiques de Toulouse]{Lorick Huang} 
\addresstwo{INSA de Toulouse, IMT UMR CNRS 5219, Université de Toulouse, 135 avenue de Rangueil 31077 Toulouse Cedex 4 France} %
\emailtwo{lhuang@insa-toulouse.fr} 

\authorthree[Institut Mathématiques de Toulouse]{Monique Pontier} 
\addressthree{Institut de Mathématiques de Toulouse. CNRS UMR 5219. Université Paul Sabatier, 118 route de Narbonne, F-31062 Toulouse cedex 09.} 
\emailthree{monique.pontier@math.univ-toulouse.fr} 

\begin{abstract}

In a previous work \cite{coutin:pontier:2020}, it was shown that the joint law of a diffusion process and the running supremum of its first component is absolutely continuous, and that its density satisfies a non standard weak partial differential equation (PDE). In this paper, we establish the uniqueness of the solution to this PDE, providing a more complete understanding of the system’s behavior and further validating the approach introduced in \cite{coutin:pontier:2020}.
\end{abstract}

\keywords{boundary condition; joint law; weak uniqueness; running supremum}

\ams{60H30}{35D30}

\section{Introduction}

The study of the joint distribution of a stochastic process and its running supremum has important applications across various domains, particularly in mathematical finance, where it plays a key role in pricing exotic options such as lookback and barrier options. These options depend on the extreme behavior of the underlying asset over time, making an understanding of joint distributions crucial for accurate pricing and effective risk management.

Assets are typically modeled using stochastic differential equations of the form:
\begin{eqnarray*}
dX_t &=& B(X_t) dt + \sigma(X_t) dW_t, \quad X_0 \sim f_0, \\
M_t &=& \sup_{s \le t} X_s^1,
\end{eqnarray*}
where \( X_0 \) is a random variable, independent of \( W \), with density \( f_0 \). In this equation, $X$ is $\R^d$-valued, and $X_s^1$ represents its first component.

For instance, for a fixed maturity time $T$, the payoff of a lookback put option with a floating strike depends on the running supremum of the asset price and is given by:
$$
\mbox{Payoff} = M_T-X_T.
$$
This setup leads to the challenge of characterizing the joint law of  $(X_T,M_T)$, which is non-trivial as it depends on the entire path of the asset  $(X_t)_{t \le T}$.

In the literature, the pricing of a lookback option with a floating strike is available in the special case where $X$ follows a geometric Brownian motion (see, e.g., Musiela and Rutkowski \cite{Musiela:Rutkowski}). However, pricing such options is significantly more complex than for European options due to the path dependency.

Many studies, starting as early as \cite{CsFolSal}, have focused on understanding the joint distribution of \( (V_t)_{t \ge 0} = (M_t, X_t)_{t \ge 0} \). These works primarily aimed at theoretical insights, such as the regularity of hitting or local times. While these results have significantly advanced our understanding, their direct applicability to practical problems, such as pricing or hedging, can be limited.

For example, \cite{BHR} provides sharp upper and lower bounds, which are useful but leave room for improvement. Continuous efforts have been made to refine our understanding of the joint distribution. One such attempt is found in \cite{dorobantu}, which offers an approximation to speed up Monte Carlo simulations. Hayashi and Kohatsu-Higa \cite{hayashi:kohastu:2013} further demonstrated that the joint law of \( (V_t)_{t \ge 0} \) is absolutely continuous, a result expected in analogy to the well-known case of Brownian motion and its running supremum. In the Brownian setting, an explicit density formula has existed since \cite{levy}, with more recent treatments in \cite{CJY}. However, these arguments rely on the reflection principle, which does not extend to general diffusions.

In a series of papers \cite{coutin:pontier:2019} and \cite{coutin:pontier:2020}, the authors established that the distribution of \( (M_t, X_t) \) is absolutely continuous, admits a series expansion and satisfies a weak PDE. For  \( \Phi \)  a test function belonging to a suitable class of functions, and using the notation  $(m,x)=(m,x^1,\tilde{x})\in \R^{d+1}$, the PDE can be expressed as:
\begin{eqnarray}
\label{edp1m}
&&\int_{\cal T}\Phi(m,x)p(m,x;t)dmdx 
=\int_{\mathbb{R}^d} \Phi(m,m,\tilde{x})f_0(m,\tilde{x})dmd\tilde{x} \\
&+& \int_0^t\int_{\cal T} p(m,x;s) \mathcal{L} \Phi(m,x) \, dmdx \, ds 
+ \frac{1}{2} \int_0^t\int_{\mathbb{R}^d} \frac{\partial \Phi}{\partial m} (m,m,\tilde{x}) p(m,m,\tilde{x};s) \, dmd\tilde{x} \, ds. \nonumber
\end{eqnarray}

The first term $\Phi(m,m,\tilde{x})$ corresponds to the case where $X_0^1 =M_0$, the second term is the classical contribution, and the third term comes from the treatment of $\E [ \int_0^t \frac{\partial}{\partial m} F(M_s, X_s) dM_s]$ and is related to $dM_s$.
As a local-time term, this contribution only increases during times $s$ where $M_s = X_s$ and $M_s$ is strictly increasing (see Theorem 2.3 of \cite{coutin:pontier:2020} for more details). 
In addition,  \( \mathcal{T} \) denotes the subset of \( \mathbb{R} \times \mathbb{R}^d \) defined by:
\[
\mathcal{T} := \{ (m,x) \in \mathbb{R} \times \mathbb{R}^d \mid m \ge x^1 \} \subset \R^{d+1}
\]
with \( \mathcal{L} \) being the generator of \( (X_t)_{t \ge 0} \) acting on the variable \( x \) as:
\[
\mathcal{L} f(x) = B(x) \cdot \nabla f(x) + \frac{1}{2} \Delta f(x).
\]

In this paper, we extend the operator \( \mathcal{L} \) to functions \( \Phi \) in \( C^2(\mathbb{R}^{d+1}, \mathbb{R}) \) as:
\[
\mathcal{L}(\Phi)(m,x) = \sum_{i=1}^d B^i(x) \partial_{x^i} \Phi(m,x) + \frac{1}{2}\sum_{i=1}^d  \partial_{x^i x^i}^2 \Phi(m,x),
\]

The main contribution of this paper is to prove the uniqueness of the solution to the PDE \eqref{edp1m} in a suitable functional space \( \mathcal{X} \), defined below. As a result, the series expansion obtained in \cite{coutin:pontier:2020} is shown to be the unique solution in a weak sense to the PDE \eqref{edp1m}. 



\begin{remark}[Notations]
Throughout this paper we will use the following notations
\begin{itemize}
\item $x=(x^1,\tilde{x})$ when no ambiguity is possible, in order to single out the behavior on the first component. 
\item Similarly, we write $dx=dx^1 d\tilde x$ to shorten the notations.
\item In $\R^{d+1}$, the notation $(m,x)\in \R^{d+1}$ is the same as $(m,x^1, \tilde{x})$.
\item  Let us denote $C^k_c(\R^d)$ the functions $C^k$ class  with compact support, 
and $C^k_b(\R^d)$ the bounded functions $C^k$ class { with bounded derivatives}, for $k=1,2$.

\item Let us denote $p_V$    the density of the probability law of process $V_t=(M_t,X_t)$, whose existence has been established in \cite{coutin:pontier:2019}.

\item In \eqref{edp1m} we write $\frac{\partial \Phi}{\partial m} (m,m,\tilde{x})$ to mean $\frac{\partial \Phi}{\partial m} (m,x) \1_{m=x^1} = \frac{\partial \Phi}{\partial m} (m,x^1,\tilde{x}) \1_{m=x^1}$.

\item Throughout the paper, we fix a probability space $(\Omega, \mathcal{F},\PP)$ in which all our random variables are defined, and we use the notation $E$ for the expectation with respect to $\PP$. In Proposition \ref{lem-semi-groupe-dual}, we change probability with Girsanov's Theorem, and therefore, we will write $E_\PP$ and $E_\mathbb{Q}$ to denote the expectations with respect to $\PP$ and $\mathbb{Q}$ respectively.

\end{itemize}
\end{remark}

\begin{definition} 
\label{def-espace-sol}
Let $T>0.$ 
Define ${\cal X}$ to be the space  of { the { real}
functions on 
$]0,T]\times{\cal T}$} satisfying
 the following items:
\begin{itemize}
\item[(a)]
{ $\sup_{t \in {]}0,T]}
\int_{{\mathbb R}^d }\left[\int_{m\geq x^1}|p(m,x;t)|dm\right]^2 dx <\infty$}.


\item[(b)] The application 
$ (m,\tilde x;s)\mapsto p(m,m,\tilde x;s)\in L^1({]0,T]};L^2(\R^d))$

\item[(c)]
For all $t\in ]0,T],$  $(m,\tilde x)\mapsto \sup_{u>0}p(m,m-u,\tilde x;t)\in L^1(\R^d);$ 
{ and for  almost surely all $(m,\tilde{x}) \in {\mathbb R}^d,$ 
{ $x^1 \mapsto p(m,x^1,\tilde{x};t)$ is continuous on $]-\infty, m[$ } and  $\lim_{u\rightarrow 0^+} p(m,m-u,\tilde{x};t)$
 exists and is denoted by $ p(m,m,\tilde{x};t)$}.

%
\end{itemize}
\end{definition}

\begin{remark}

In the case of a smooth solution and in dimension 1, we can perform an integration by parts and reformulate PDE \eqref{edp1m} as:
$$
\begin{cases}
\partial_t p(m,x;t) = {\cal L}^*p(m,x;t)  & \mbox{ if } m >x \\
B(m)p(m,m;t) = \frac{1}{2}(\partial_1 + \partial_2)p(m,m;t) + \frac{1}{2} \partial_2 p(m,m;t) & \mbox{ if } m=x.
\end{cases}
$$
This PDE is non-standard in the sense that the operator ${\cal L}^*$ only acts on the variable $x\in \R$. 
The Dirichlet condition appears as a reformulation of the $\frac{\partial \Phi}{\partial m}$ term in PDE \eqref{edp1m} and is not a source term, as it is imposed by the PDE.

\end{remark}

\begin{theorem}
\label{mthm}
Assume that $B \in C^1_b(\R^d, \R^d).$
Let $T>0,$ and $f_0 \in L^1(\R^d)\cap L^2(\R^d)$ with positive values such that $\int_{\R^d} f_0(x) dx=1.$ Then, 
$p_V$ belongs to ${\cal X}$ and is the unique solution in ${\cal X}$ of~\eqref{edp1m} in a weak sense.
\end{theorem}
{
 It should be emphasized that this
 PDE is set in dimension $d+1$, but it is
  degenerate, in the sense that there are no derivative related to the component
    $m.$ Moreover, the presence of a boundary condition takes us away from the classical equations investigated in textbooks such as Evans \cite{evans}.}

During the preparation of this work, we became aware of \cite{Frikha:li:2020}, in which the authors skillfully apply the parametrix technique to derive a series expansion (the same as \cite{coutin:pontier:2020}), density estimates and prove uniqueness for the martingale problem in a class of one-dimensional path-dependent SDEs. 
A key insight of \cite{Frikha:li:2020} is the use of appropriate transmission conditions to compute the generator, with these conditions satisfying a convolution-like property relative to a certain measure. Although \cite{Frikha:li:2020} focuses on the one-dimensional case, their model accounts for a diffusion coefficient. We believe that their arguments could extend to higher dimensions, albeit with more complex notations and technical challenges.

However, a key distinction between \cite{Frikha:li:2020} and our work lies in the treatment of the martingale problem, which provides weak uniqueness but implicitly assumes that the solution is a probability measure. In contrast, our proof does not require non-negativity. Additionally, building on the results from \cite{coutin:pontier:2020}, our boundary condition naturally arises from Itô’s formula.

Moreover, while the parametrix technique can be applied to analyze the regularity of the density, Frikha and Li 
\cite{Frikha:li:2020} primarily focus on the well-posedness of the martingale problem. In contrast, our paper, viewed as a continuation of \cite{coutin:pontier:2019, coutin:pontier:2020}, recovers the regularity of the density in dimension one, as a consequence of uniqueness.


%

This paper is organized as follow. In Section \ref{unique}, we prove that the solution of \eqref{edp1m} with fixed initial condition is unique in a certain class of function ${\cal X}$. 
To achieve this, we deduce from the PDE \eqref{edp1m} an auxiliary parabolic problem, with source term coming from our boundary condition. We then use Ball’s Theorem \cite{ball:1977} on the solution of that parabolic problem and using a localization argument, we show that the trace term is solution of a Volterra integro-differential equation with zero initial condition.
The fact that our auxiliary parabolic problem falls into the framework of Ball's Theorem \cite{ball:1977} is checked in Section \ref{checkingBall}. 
Then Section \ref{pVdansX} proves that $p_V$ is an element of set ${\cal X}.$ Finally 
an appendix with some tools is provided.

\section{Uniqueness for weak solution of PDE in the set ${\cal X}$}
\label{unique}
The aim of this section is to prove the following theorem.
\begin{theorem}
\label{unique0}  Let $T > 0$ and $B \in C^1_b(\mathbb{R}^d, \mathbb{R}^d)$. In the space $\mathcal{X}$ defined in Definition \ref{def-espace-sol}, the PDE \eqref{edp1m} with initial condition $f_0 \in L^1(\mathbb{R}^d) \cap L^2(\mathbb{R}^d)$ has at most one solution.
\end{theorem}


{ In this section, we develop the tools needed to establish uniqueness for the solution of the PDE \eqref{edp1m}. 
Uniqueness is understood in a weak sense: if $p_1$ and $p_2$ are two weak solutions of \eqref{edp1m} belonging to ${\cal X}$, then they are equal almost everywhere. 
To do so, we consider $p_1$ and $p_2$, two weak solutions to \eqref{edp1m} in ${\cal X}$ with the same initial condition, and show that their difference is zero almost surely.
Having the same initial condition, for $q=p_1-p_2$, we derive the following equation:
\begin{eqnarray}
\label{prep:EDP}
\int_{\cal T}\Phi(m,x)q(m,x;t)dmdx &=&
\int_0^t\int_{\cal T} q(m,x;s){\cal L}\Phi(m,x)dmdxds\nonumber
\\
&+ &{ \frac{1}{2}}
\int_0^t\int_{\R^d} \frac{\partial \Phi}{\partial m}(m,x )\textbf{1}_{m=x^1}  q(m,m,\tilde x;s)dmd\tilde xds
\end{eqnarray}
for all test function $\Phi \in C^{{ 2}}_c(\R^{d+1}).$
We can chose in particular $\Phi(m,x)=H(m)F(x)$, for $(m,x)\in {\cal T},$ where  $H\in C^1_c(\R)$ and $F\in C^2_c(\R^d)$. And since there is no second derivative with respect to $m,$ we introduce 
\begin{equation}
\label{def:q_h}
q_H(x;t) := \int_{m \ge x^1} H(m) q(m,x;t) dm,~~x \in \R^d,~~t\in [0,T]
\end{equation}
where $H\in C^1_b(\R)$.

The rest of the proof is as follows.
We show in Proposition \ref{equation_qH} that $q_H$ is solution to a parabolic PDE with a source term on $L^2(\R^d)$. Using Ball theorem in $L^2$, we prove equation \eqref{prelim:H} below to get the following representation for $q_H$:
$$
q_H(x;t)=\frac12 \int_0^t \int_{\R^d} H'(y^1) q(y^1, y;s) \Gamma(x,y;t-s) ds dy.
$$
Using a density argument,  we localize on the diagonal and derive \eqref{qonDiag}:
$$
q(y^1,y^1,\tilde{x};t)   { =-}  \frac12 \int_0^t \int_{\R^{d-1}}\partial_{x^1} \Gamma(y^1,\tilde{x},y^1,\tilde{y},t-s)q(y^1,y^1,\tilde{y};s) d \tilde{y} ds,
$$
where $\Gamma$ is a kernel defined in \eqref{rep-gamma-gar}.
The estimates on the kernel $\Gamma$ then show that $q$ is zero almost everywhere.


\subsection{The function $q_H$ is solution of a parabolic partial differential equation}

In this section, we derive from \eqref{edp1m} the auxiliary parabolic problem \eqref{edp-qH} with a source term, whose solution is obtained by applying a result from Ball \cite{ball:1977}. We state this as the following proposition.
\begin{prop}
\label{equation_qH}
Let $q \in {\mathcal X}$ be a solution of \eqref{edp1m} with a null initial condition, and let $H \in C^{2}_{b}(\mathbb{R}, \mathbb{R})$ with a compactly supported derivative. The function $q_H(\cdot; \cdot)$ satisfies the equation in a weak sense:
\begin{align}\label{edp-qH}
q_H(x;t) = \int_0^t {\cal L}^* q_H(x;s) ds + \frac12 \int_0^t  H'(x^1) q(x^1,x^1,\tilde{x};s) ds,~~t\in [0;T]
\end{align}
where ${\cal L}^*$ is
\begin{align*}
{\mathcal L}^*(\Phi) (x)= -B(x)\cdot\nabla \Phi(x) { -} \Phi(x) {\rm div} B(x) + {\frac12}\Delta \Phi (x).
 \end{align*}
 { Moreover,~~for all $t\in [0,T],$ $q_H(.;t) \in L^2(\R^d).$}
\end{prop}
%
\begin{proof}

Let $H\in C_b^2({\mathbb R},{\mathbb R})$ with compactly supported derivative.
Let $q \in {\mathcal X}$ such that $q(., 0)=0$.

We start by proving that 
$ q_H(.;t) \in L^2({\mathbb R}^d),$ for all $t\in [0,T]. $
\\
Recall from Item (a) of Definition \ref{def-espace-sol} that $$\sup_{t \in ]0,T]}\int_{\R^d}\left[\int_{\R} |q(m,x;t)| dm\right]^2 dx <\infty,$$
and the fact that $H$ is bounded yields
\begin{equation}\label{borne:H:q}
\int_{\R^d} \left[ \int_{\R} |H(m) | |q(m,x;t) | dm \right]^2 dx <\infty,~~\forall t >0.
\end{equation}
Hence $q_H (.;t) \in L^2(\R^d).$ 

Next, we proceed with the proof of \eqref{edp-qH}. The expression \eqref{edp-qH} in a "weak sense" means that it should be understood as follows: for all functions $F \in C^2_c({\mathbb R}^d, {\mathbb R})$ and for all $t \in [0, T]$,
\begin{align}\label{edp-qh-test}
  \int_{{\mathbb R}^d} F(x) q_H(x; t) \, dx &= \int_0^t \int_{{\mathbb R}^d} {\mathcal L} F(x) q_H(x; s) \, dx \, ds \nonumber\\ & + \frac{1}{2} \int_0^t \int_{{\mathbb R}^d} H'(m) F(m, \tilde{x}) q(m, m, \tilde{x}; s) \, dm \, d\tilde{x} \, ds.
\end{align}
We start from Equation \eqref{prep:EDP} satisfied by $q$, and substitute for $F \in C^2_c(\mathbb{R}^d)$:
$$\Phi(m, x) = H(m) F(x) \mathbf{1}_{x^1 \le m}.$$ 
\\
(i) \underline{Left-hand side}: Since $F$ has compact support and $H$ is bounded, we can apply Fubini's Theorem to exchange the order of integration (recall equation \eqref{borne:H:q} to justify the finiteness of the integrals). The left-hand side then becomes:

\begin{eqnarray*}
\int_{\cal T}\Phi(m,x)q(m,x;t)dmdx &=& \int_{\R^d} F(x)  \left( \int_{m \ge x^1} H(m) q(m,x;t) dm \right) dx\\
&=& \int_{\R^d} F(x) q_{H}(x;t) dx.
\end{eqnarray*}
(ii) \underline{Right hand side}: Now considering the right hand side of \eqref{prep:EDP}, we have for the first contribution:
$$
\int_0^t\int_{\cal T} q(m,x;s){\cal L}\Phi(m,x)dmdxds = \int_0^t\int_{\cal T} q(m,x;s){\cal L}\Big( H(m) F \Big)(x) dmdxds.
$$
Note that the operator ${\cal L}$ acts only on  the variable $x$. We thus have:
\begin{eqnarray*}
\int_0^t\int_{\cal T} q(m,x;s){\cal L}\Phi(m,x)dmdxds 
&=& \int_0^t\int_{\R^d} \int_{m \ge x^1} q(m,x;s) H(m) {\cal L} F (x) dmdxds \\
&=&\int_0^t \int_{\R^d} q_H(x;s) {\cal L} F(x) dx ds .
\end{eqnarray*}
(iii) Finally we consider the second term in the right hand side of \eqref{prep:EDP} , since 
$$ \frac{\partial \Phi}{\partial m}(m,m,\tilde x ) = \frac{\partial \Phi}{\partial m}  (m,x^1,\tilde{x})\textbf{1}_{m=x^1},$$ 
meaning: we first compute the derivative of $\Phi$ with respect to the first component, and evaluate at $(m,m,\tilde x) \in \R \times \R^d$. The 
second term in the right hand side of \eqref{prep:EDP} becomes  up to the factor $\frac12$:

\begin{eqnarray*}
&&\int_0^t\int_{\R^d} \frac{\partial}{\partial m}\Phi(m,m,\tilde x ){  q(m,m,\tilde x;s)}dmd\tilde xds\\
&=& \int_0^t\int_{\R^d} H'(m) F(m,\tilde x){ q(m,m,\tilde x;s)}dmd\tilde xds
\\
&=&\int_{\R^d} \left( \int_0^t H'(m) q(m,m,\tilde x;s) ds \right) F(m,\tilde{x}) dmd\tilde x \\
&=& \int_0^t\int_{\R^d} H'(m) F(m,\tilde x){ q(m,m,\tilde x;s)}dmd\tilde xds.
\end{eqnarray*}
 %

Consequently, we have proved that $q_H$ satisfies   equation \eqref{edp-qh-test}:
 \begin{align*}
 \int_{\R^d} F(x) q_H(x;t) dx= \int_0^t \int_{\R^d} {\mathcal L}F(x) q_H(x;s) dxds
 +\frac12\int_0^t \int_{\R^d} H'(x^1){ F(x)}q(x^1,x;s)dxds.
 \end{align*}

 \end{proof}

\begin{remark}
We prove in Proposition  \ref{lem-semi-groupe-dual} below that the
adjoint operator 
$${\cal L}^*f (x) =  -B(x) \cdot \nabla f(x)-f(x){ \sum_{k=1}^d\partial_{x^k} B^k(x)}+{\frac12}\Delta f(x), $$
 is a densely defined closed operator and is the  generator of the semi-group 
 $Q$  with kernel $\Gamma$:
\begin{align}
\label{semi-groupe-dual}
 Q_t(f)(x):=\E\left[f(X^x_t)\exp\left(-\int_0^t\sum_{k=1}^d \partial_{x^k} B^k(X^x_u)du\right)\right]=\int f(y)\Gamma(x,y;t)dy
\end{align}
for $0 \le t \le T$, $x\in \R^d$ and $f \in L^2(\R^d)$ . 
The kernel $\Gamma$ has Gaussian estimates (see Garroni and Menaldi \cite{garroni:menaldi:1992}).

\end{remark}

Now, let us quote the following theorem 
due to Ball \cite{ball:1977}.

\begin{theorem} 
\label{ball}
Let $A$ a densely defined closed linear operator on a Banach space $X$ which generates a strongly continuous semi-group $Q$
on $X$ { of bounded linear operators $Q_t$ bounded on $X$}. Let $f$ belonging to $L^1([0,T], X)$ and $u\in C([0,T],X)$  be a weak solution of the PDE:
\begin{equation}
\label{ball:eq}
u'(t)=Au(t)+f(t),~t\in (0,T],~~u(0)=x\in D(A)\subset X,
\end{equation}
where $D(A)$ is the domain of operator $A$ In that case, the unique  solution of \eqref{ball:eq} is expressed as
\begin{equation}
\label{ballSol}
u(t)=Q(t)x+\int_0^t Q(t-s)f(s)ds,~\forall t\in [0,T].
\end{equation}	
\end{theorem}

We will verify that the assumptions of Ball’s theorem are satisfied in our setting in Section~\ref{checkingBall} below. We apply Ball’s theorem with the following choices: 
$X := L^2(\mathbb{R}^d)$, $A := {\cal L}^*$, $u(t) := q_H(\cdot; t)$, $f : t \mapsto \frac{1}{2} H'(x^1) q(x^1, x^1, \tilde{x}; t)$, and $x = 0$.
The conclusion of the theorem in our context is stated as follows:

\begin{corollary} 
For all $H \in C^2_b(\mathbb{R})$ with compactly supported derivatives and for all $t \in [0, T]$, $dx$-almost surely,
\begin{equation}
\label{prelim:H}
q_H(x; t) = \int_\mathbb{R} H(m) q(m, x; t) \, dm = \frac{1}{2} \int_0^t \int_{\mathbb{R}^d} H'(y^1) q(y^1, y; s) \Gamma(x, y; t - s) \, ds \, dy,
\end{equation}
where $\Gamma$ is the kernel defined in \eqref{semi-groupe-dual}.
\end{corollary}

The existence of $\Gamma$ is established in Proposition \ref{lem-semi-groupe-dual} below.

From here, we proceed to show that $q = 0$ almost surely. The main challenge in the calculation is relating $q(m, x; t)$ on the left to its value on the diagonal $m = x^1$, which is the quantity appearing on the right-hand side of \eqref{prelim:H}. To establish this connection, we employ a localization argument in the next section, leveraging a carefully selected test function $H$.

\subsection{Localisation on the diagonal}
We can prove the following result:
%

\begin{prop}
\label{identite:q:prelim}
{ Let $B \in C^1_b(\R^d,\R^d)$ and $q\in {\mathcal X}$ be a solution of \eqref{edp1m} with null initial condition.}
The following identity holds
{ $\forall t\in [0,T]$ and $dy^1 d\tilde{ x}$ almost surely}:
 \begin{equation}
 \label{qonDiag}
q(y^1,y^1,\tilde{x};t)   { =-}  \frac12 \int_0^t \int_{\R^{d-1}}\partial_{x^1} \Gamma(y^1,\tilde{x},y^1,\tilde{y};t-s)q(y^1,y^1,\tilde{y};s) d \tilde{y} ds.
\end{equation}

\end{prop}

We break down the proof of the above identity in two steps. 

{\bf First step:} Going back to \eqref{prelim:H}, we multiply both sides by $F\in C_c(\R^{d})$ and integrate in $ dxdm$ on the set ${m>x^1}\subset\R^{d+1}$:
\begin{equation}
\label{prelim:TH}
\int_{\R^{d+1}}{\mathbf 1}_{\{m>x^1\}} F( x)H(m)  q(m,x;t) dm dx ={\frac12}\int_0^t \int_{\R^{2d}} F( x)H'(y^1) q(y^1, y;s) \Gamma(x,y;t-s) ds dy dx.
\end{equation}

The aim of this first step is to prove the following decomposition:
\begin{lemma} 
\label{preuvediagball}
Let $H$ be a $C^2_c(\R)$ and $F \in C_c({\mathbb R}^{d-1})$.
It holds that {  $\forall t\in [0,T]$}:
 \begin{eqnarray}
 &&\int_{{\mathbb R}^d} F(\tilde x)H(m) q(m,m,\tilde{x};t) dm d\tilde{x}=\nonumber\\
 \label{diag-ball1}
 && \frac{1}{2}\int_0^t \int_{{\mathbb R}^{2d-1}}  F(\tilde x)H'(y^1) \Gamma(y^1,\tilde{x},y^1,\tilde{y};t-s){ q(y^1,y^1,\tilde{y};s)} ds dy^1 d\tilde{y} d\tilde{ x}
 \\
 &&  -\frac{1}{2}\int_0^t \int_{{\mathbb R}^{2d-1}} F(\tilde x) H(y^1)\partial_{x^1} \Gamma(y^1,\tilde{x},y^1,\tilde{y};t-s)q(y^1,y^1,\tilde{y};s) ds dy^1 d\tilde{y} d\tilde{ x}.
 \label{diag-ball2}
 \end{eqnarray}
\end{lemma}
Note that in Lemma \ref{preuvediagball}, the functions $\Gamma$ and $q$ are considered on the diagonal $m=y^1$, as opposed to Equation \eqref{prelim:TH}.
The proof strategy involves localizing the variable $m$ in \eqref{prelim:TH} around $x^1$ by means of an identity approximation.
\begin{proof}[Proof of Lemma \ref{preuvediagball}]
By using Item (c) of Definition \ref{def-espace-sol}, observe that
$x^1 \mapsto q_H(x^1,\tilde{x};t)$ is continuous
$d\tilde{x}$ almost surely.

(i) Localizing:
{ According to Lemma \ref{lem-cont-gamma}  $d\tilde{x}$ almost surely  
$$x^1 \mapsto \int_0^t \int_{\R^d} H'(y^1) q(y^1,y,s) \Gamma(x,y;t-s) dyds \mbox{ is continuous.}$$ } { We are in position to localize in $x^1$ using the following compact support approximation of unity.}
%
{ Let $\chi$ be the $C^{\infty}$ function with compact support in $[-1,1]$ defined by $$\chi(u)=
\begin{cases}
~~a \exp \left(-\frac{1}{1-u^2}\right) & \mbox{ if }|u|<1 \\
~~0 & \mbox{else}.
\end{cases}
$$

Then, $\chi$ is non negative, even,  and { $a$ is chosen such that}
$\int_{\R} \chi(u) du =1.$ }

For fixed $x^1 \in \R$,  we consider:
{ $$
\chi_\varepsilon(m) = \frac{1}{\varepsilon} \chi\left( \frac{m-x^1}{\varepsilon}\right). 
$$}
From properties of ${ \chi}$, we can derive:
\begin{equation}
\label{convergence:dirac}
\lim_{\varepsilon\rightarrow 0}\chi_\varepsilon(m)dx^1=\delta_m(dx^1).
\end{equation}
\begin{remark}
 This convergence, which enables the localization of $q$ around the diagonal, is rigorously established in our context using the Lebesgue Dominated Convergence Theorem. For more details, we refer to Lemmas \ref{lemme:approx:dirac} and \ref{lem-cont-gamma-2} in Section \ref{technical:section}. In particular, Lemma \ref{lemme:approx:dirac}, in conjunction with Lemma \ref{lem-cont-gamma}, leads to expression \eqref{diag-ball1}, while Lemma \ref{lem-cont-gamma-2} yields expression \eqref{diag-ball2}.
 
\end{remark}
%
%

Let us now consider equation \eqref{prelim:TH} with $m\mapsto \chi_\varepsilon(m) H(m)$ instead of $H$. For $F\in C_c(\R^{d-1})$, in other words, independent of $x^1$, we get:
\begin{eqnarray}\label{prelim:TH-modif}
&&\int_{\R^{d+1}} F(\tilde x){\1_{\{x^1<m\}}} \Big( \chi_\varepsilon(m) H(m) \Big)  q(m,x;t) dm dx
\\
&=&{\frac12}\int_0^t \int_{\R^{2d}} F(\tilde x)\Big(\chi_\varepsilon(y^1) H(y^1) \Big)' q(y^1, y;s) \Gamma(x,y;t-s) ds dy dx.\nonumber 
\end{eqnarray}
On the left hand side {of \eqref{prelim:TH-modif}}, letting $\varepsilon \rightarrow 0$, from \eqref{convergence:dirac} { and using the continuity $x^1\mapsto q(m,x^1,\tilde x;t)$ for all { almost } $(m,\tilde x)$} { and all $t$} (see third part of Item (c) of Definition \ref{def-espace-sol}),  we get: 
$$
\lim_{\varepsilon \rightarrow 0}\int_{\R^{d+1}} F(\tilde x) {\1_{\{x^1<m\}}}\Big( \chi_\varepsilon(m) H(m) \Big)  q(m,x;t) dm dx ={\frac12}\int_{\R^{d}} F(\tilde x)  H(m)   q(m,m,\tilde{x};t) dmd\tilde x.
$$
On the right hand side { of \eqref{prelim:TH-modif}}, we compute the derivative:
\begin{eqnarray}
\label{detail}
&&\int_0^t \int_{\R^{2d}} F(\tilde x)\Big(\chi_\varepsilon(y^1) H(y^1) \Big)' q(y^1, y;s) \Gamma(x,y;t-s) ds dy dx\nonumber
\\
&=&\int_0^t \int_{\R^{2d}} F(\tilde x)\chi_\varepsilon(y^1) H'(y^1)  q(y^1, y;s) \Gamma(x,y;t-s) ds dy dx\\
&&+\int_0^t \int_{\R^{2d}} F(\tilde x) \chi_\varepsilon'(y^1)H(y^1)q(y^1, y;s) \Gamma(x,y;t-s) ds dy dx.\nonumber
\end{eqnarray}
(ii) We now let $\varepsilon \rightarrow 0$ respectively in the two terms in \eqref{detail}.
{
According to  Garroni-Menaldi 
\cite{garroni:menaldi:1992} (3.35) page 187:
\begin{align}
\label{rep-gamma-gar}
\Gamma(x,y;t) =\Gamma_0(x,y;t) + \Gamma_1(x,y;t)
\mbox{ where }
\Gamma_0(x,y;t) =\frac{e^{-\frac{\|x-y\|^2}{2t}}}{\sqrt{2\pi t}^d},
\end{align}
and for all $y \in R^d, ~~t>0$ $x\mapsto \Gamma_1(x,y;t) \in C^1(\R^d). $ 
Moreover, from Lemma 3.3 page 184, estimation (3.25) of Garroni-Menaldi \cite{garroni:menaldi:1992}, for $\alpha \in ]0,1[$  there exist some positive constants $C$ and $c$ such that
\begin{align}
\label{maj-garro-gamma-1}
\left| \Gamma_1(x,y;t)\right| \leq Ct^{-\frac{d}{2} +\frac{\alpha}{2}} e^{-\frac{\|x-y\|^2}{ct}},
\end{align}
\begin{align}
\label{maj-garro-der-gamma-1}
\left| \partial_{x^l}\Gamma_1(x,y;t)\right| \leq C{ t^{-\frac{d+1}{2}+\frac{\alpha}{2} }} e^{-\frac{\|x-y\|^2}{ct}}.
\end{align}

Now, letting $\varepsilon \rightarrow 0$ 
and  using Lemma \ref{lemme:approx:dirac},  and the continuity of
 $x^1\mapsto \Gamma(x^1,\tilde x, y;t)$,} $\forall t\in [0,T]$
 (see Theorem 3.5 in \cite{garroni:menaldi:1992}, page 186)
 the  first   term  on the right hand side of \eqref{detail} converges to:
\begin{eqnarray*}
\int_0^t \int_{\R^{2d}} F(\tilde x)\chi_\varepsilon(y^1) H'(y^1)q(y^1, y;s) \Gamma(x,y;t-s) ds dy dx,
\\
{ \underset{\varepsilon \rightarrow 0}{\longrightarrow}}\int_0^t \int_{\R^{2d-1}} F(\tilde x) H'(y^1)q(y^1,y^1, \tilde{y} ;s) \Gamma(y^1,\tilde{x},y^1,\tilde{y};t-s) ds dy d\tilde{x},
\end{eqnarray*}
which gives \eqref{diag-ball1} in Lemma \ref{preuvediagball} above.
\\
(iii) For the  { second term}
in \eqref{detail} we look at the integral in $dy^1$, we need to integrate:
\begin{eqnarray*}
&&\int_0^t \int_{\R^{2d}} F(\tilde x)\chi_\varepsilon'(y^1) H(y^1)  q(y^1, y;s) \Gamma(x,y;t-s) ds dy dx.
\end{eqnarray*}
%
%
Then, since
$\chi_\varepsilon(y^1) =  \frac{1}{\varepsilon} \chi( \frac{y^1-x^1}{\varepsilon}),$ we have: 
$\int_{{\mathbb R}} \chi_\varepsilon'(y^1)\Gamma_0(x,y;t-s) dx^1=0$, and
\begin{eqnarray*}
&&\int_0^t \int_{\R^{2d}} F(\tilde x)\chi_\varepsilon'(y^1) H(y^1)  q(y^1, y;s) \Gamma(x,y;t-s) ds dy dx\\
=
&&\int_0^t \int_{\R^{2d}} F(\tilde x)\chi_\varepsilon'(y^1) H(y^1)  q(y^1, y;s) \Gamma_1(x,y;t-s) ds dy dx.
\end{eqnarray*}
We can perform an integration by parts over $dx^1$ on $\R$:
\begin{eqnarray*}\int_\R  \chi'_\varepsilon(y^1) \Gamma_1(x^1,\tilde x,y^1,\tilde y;t-s)dx^1
&=&\int_\R  \left(\frac{1}{\varepsilon} \chi( \frac{y^1-x^1}{\varepsilon}) \right)' \Gamma_1(x^1,\tilde x,y^1,\tilde y;t-s)dx^1\\
&=& -\int_\R\chi_\varepsilon(y^1)\partial_{x^1}\Gamma_1(x^1,\tilde x,y^1,\tilde y;t-s)dx^1
\end{eqnarray*}
since when $x^1\rightarrow\infty,$   $\chi_\varepsilon(x^1) $ goes to $0$. 

Letting now $\varepsilon\rightarrow 0$, { using estimation \eqref{maj-garro-der-gamma-1}}, we obtain: 
$$\int_\R\chi_\varepsilon(y^1) \partial_{x^1}\Gamma_1(x^1,\tilde x,y^1,\tilde y;t-s)dx^1
\underset{\varepsilon \rightarrow 0} \longrightarrow \partial_{x^1}
\Gamma_1(y^1,\tilde x,y^1,\tilde y;t-s).$$
{ Note that according to the definition of $\Gamma,$  $\Gamma_1$ and $\Gamma_0$ (see \eqref{rep-gamma-gar})
$$\partial_{x^1}\Gamma_1(m,\tilde x,m,\tilde y;t-s)=\partial_{x^1}\Gamma(m,\tilde x,m,\tilde y;t-s).$$}
Consequently,  {using Lemma \ref{lem-cont-gamma-2}}, we have {  $\forall t\in [0,T]$:
\begin{eqnarray*}
&&\int_0^t \int_{\R^{2d}} F(\tilde x)\chi_\varepsilon'(y^1) H(y^1)  q(y^1, y;s) \Gamma(x,y;t-s) ds dy dx
\\
&&\underset{\varepsilon \rightarrow 0}
 \longrightarrow{ -}\int_0^t \int_{\R^{2d-1}} F(\tilde x) H(y^1)  q(y^1, y;s)\partial_{x^1}
\Gamma(y^1,\tilde x,y^1,\tilde y;t-s) ds dy dx,
\end{eqnarray*}}
which gives   \eqref{diag-ball2} in Lemma \ref{preuvediagball} above.

\end{proof}

\begin{remark}
We need to use dominated convergence theorem on the functions
$$
x^1\mapsto q(m,x^1,\tilde x;t)~;~\Gamma(x^1,\tilde x,m,\tilde y; t-s)~;~\partial_{x^1}\Gamma(x^1,\tilde x,m,\tilde y; t-s).
$$
in order to justify the Dirac convergence.
These functions have to satisfy the 
continuity when $x^1\rightarrow m$ and the  uniform integrability. 
The continuity is clear for $\Gamma$ and $\partial_{x^1} \Gamma$ (see e.g. \cite{garroni:menaldi:1992} Section V.3.1) and for $q$, 
{it is exactly Item (c) in Definition \ref{def-espace-sol} above}.
The integrability of $\sup_{x^1}| \Gamma(x,y;t-s)|$ and $\sup_{x^1}|\partial_{x^1} \Gamma(x,y;t-s)|$ is deduced from
 \cite{garroni:menaldi:1992} (page 171   and  (3.25) page 18).
Moreover, we also rely on the fact that $F$ has compact support in {$\R^{d-1}$}.
\end{remark}

 {\bf Second step:} 
 Going back to Lemma \ref{preuvediagball}, 
 we prove that the term \eqref{diag-ball1}:
$$
\frac{1}{2}\int_0^t \int_{{\mathbb R}^{2d-1}}  F(\tilde x)H'(y^1) \Gamma(y^1,\tilde{x},y^1,\tilde{y};t-s){ q(y^1,y^1,\tilde{y};s)} ds dy^1 d\tilde{y} d\tilde{ x}
$$
actually vanishes.

\begin{lemma}
\label{lem2}
We have {  $\forall t\in [0,T]$, $dy^1$ almost surely}
\begin{equation}
\label{RHSLem27}
\int_0^t {\int_{\R^{2d-2}} F(\tilde{x} ) q(y^1,y^1,\tilde{y};s) \Gamma(y^1, \tilde{x}, \tilde{y};t-s) dsd\tilde xd\tilde y} =0.
\end{equation}
\end{lemma}
\begin{proof}

 Let $\Psi$ a compactly supported test function and  consider the  following choice for $H$:
$$
H(m)=\int^{+\infty}_m \Psi(y^1) dy^1.
$$
In that case, Equation \eqref{prelim:H} becomes :
\begin{align*}
q_H(x;t)=\int_{{\mathbb R}^2} {\mathbf 1}_{y^1 >m>x^1}\Psi(y^1)q(m,x;t) dm dy^1= -\frac12 \int_0^t \int_{{\mathbb R}^d} \Psi(y^1) q( y^1,y;s) \Gamma(x,y;t-s) ds dy,
\end{align*}
for any  $\Psi$ so by identification $dy^1$ almost surely:
\begin{align*}
\int_{{\mathbb R}} {\mathbf 1}_{x^1 <m<y^1}q(m,x;t) dm = -\frac12 \int_0^t \int_{{\mathbb R}^{d-1}}  q( y^1,y^1,\tilde{y},
;s) \Gamma(x,y^1,\tilde{y};t-s) ds d\tilde{y}.
\end{align*}
We multiply this last identity by $-F\in C^1_k(\R^{d-1})$ and integrate with respect to $d\tilde x$:
\begin{align}
\label{**}
&-\int_{\R^{d}}F(\tilde x) {\mathbf 1}_{y^1 >m>x^1}q(m,x;t) dmd\tilde x \nonumber\\
&= \frac12 {\int_{\R^{2d-2}}}F(\tilde x)\int_0^t q( y^1,y^1,\tilde{y};s) \Gamma(x,y^1,\tilde{y};t-s) ds d\tilde{y}d\tilde x
\end{align}
and letting  $x^1$ increase to $y^1$, { using Item (a) of the Definition \ref{def-espace-sol} }
the limit of left hand is  zero { $dy^1$ almost surely.}

According to the definition of $\Gamma$ in equation \eqref{rep-gamma-gar}, estimation \eqref{maj-garro-gamma-1} and Lemma \ref{lem-cont-gamma}, the limit of the right hand side when $x^1\uparrow y^1$ is { $dy^1$ almost surely:}
\begin{align*}
\frac12 {\int_{\R^{2d-2}}}F(\tilde x)\int_0^t q( y^1,y^1,\tilde{y};s) \Gamma(y^1,\tilde{x},y^1,\tilde{y};t-s) ds d\tilde{y}d\tilde x.
\end{align*}
So \eqref{RHSLem27} is proved.
\end{proof}


\begin{proof}[Proof of Proposition \ref{identite:q:prelim}]

Going back to the conclusion of Lemma \ref{preuvediagball} and cancelling the first term by Lemma \ref{lem2}, we get {  $\forall t\in [0,T]$}:
\begin{eqnarray*}
&&\int_{{\mathbb R}^d} F(\tilde x)H(y^1) q(y^1,y^1,\tilde{x};t) dy^1 d\tilde{x}
\\
&=&{ -}\frac12  \int_0^t \int_{{\mathbb R}^{2d-1}} F(\tilde x) H(y^1)\partial_{x^1} \Gamma(y^1,\tilde{x},y^1,\tilde{y};t-s)q(y^1,y^1,\tilde{y};s) ds dy^1 d\tilde{y} d\tilde{ x},
\end{eqnarray*}
which gives by identification over $H(y^1) F(\tilde x)$ 
{$dy^1 d\tilde{ x}$ almost surely}:
$$
q(y^1,y^1,\tilde{x};t)   = { -} \frac12 \int_0^t\int_{\R^{d-1}} \partial_{x^1} \Gamma(y^1,\tilde{x},y^1,\tilde{y};t-s)q(y^1,y^1,\tilde{y};s) d \tilde{y} ds.
$$
which proves the identity \eqref{qonDiag} in Proposition \ref{identite:q:prelim}.
\end{proof}

\subsection{ Proof of Theorem \ref{unique0} }
Let us show how to conclude that $q=0$.
\\
Taking the absolute values in Proposition \ref{identite:q:prelim} gives:
\begin{align}
\label{majq}
|q(y^1,y^1,\tilde{x};t)|\leq
 \frac12 \int_0^t \int_{{\mathbb R}^{d-1}} | q( y^1,y^1,\tilde{y} ,s)\partial_{x^1} \Gamma_1(y^1,\tilde{x},y^1,\tilde{y};t-s)|ds d\tilde{ y}.
 \end{align}

 From there, the idea is to iterate \eqref{majq} on itself to prove that $|q|$ is arbitrarily small using the Gaussian bounds available for $\partial_{x^1} \Gamma_1$ (see Garroni and Menaldi \cite{garroni:menaldi:1992}). 
 In Lemma \ref{lemme3Laure}, we prove an upper bound for $|q(y^1,y^1,\tilde{x};t)|$ 
derived from the convolution properties of Gaussian densities. Then, in Lemma \ref{lemmeLaure2}, we prove that  
 $|q(y^1,y^1,\tilde{x};t)|=0$  $dtdy^1d\tilde{x}$ almost surely.

We introduce the notations:
 $$\phi_d(x;t)=\frac{1}{(\sqrt{2\pi t})^d}
 \exp-\frac{\|x\|^2}{2t};~g_\alpha(t)=t^{\alpha/2-1}
 $$
 and $g_\alpha^{*n}$ the $n$th product of convolution.
 \begin{lemma}
 \label{lemme3Laure}
 There exists  constants $C_T$ and $c>0$ such that $d\tilde xdy^1dt$ almost surely,  $\forall n$ 
 \begin{equation}
  \label{(*)}
  |q(y^1,y^1,\tilde x;t)|\leq (\frac{C_T}{2})^n\int_0^t\int_{\R^{d-1}}
  g_\alpha^{*n}(t-s_n)|q(y^1,y^1,\tilde y^n;s_n)|
  \phi_{d-1}(\tilde x-\tilde y_n;c(t-s_n))ds_nd\tilde y^n.
 \end{equation}
 \end{lemma}
 \begin{proof} We proceed by induction. Let $n=1.$ 
Equation  (3.25) page 184 \cite{garroni:menaldi:1992} with
 $l=1,0<\alpha<1$
  gives the following Gaussian estimate for $\Gamma_1$ and its derivatives:
 $$|\partial_{x^1} \Gamma_1(y^1,\tilde x,y^1,\tilde y ;t)|\leq C_T\frac{1}{\sqrt{t^{d-\alpha+1}}}\exp\left(-\frac{\| \tilde x- \tilde y\|^2}{2ct}\right).
 $$ 
 This means 
 \begin{equation}
 \label{(**)}
 |\partial_{x^1} \Gamma_1(y^1,\tilde x,y^1,\tilde y;t)|\leq C_Tg_\alpha (t)\phi_{d-1}(\tilde x- \tilde y;ct).
 \end{equation}
 We insert this bound in \eqref{majq}:
 \begin{align}
\label{majqbis}
|q(y^1,y^1,\tilde{x};t)|\leq
 \frac{C_T}{2} \int_0^t \int_{{\mathbb R}^{d-1}} \Big| q( y^1,y^1,\tilde{y} ;s) g_\alpha (t-s)\phi_{d-1}(\tilde x-\tilde y;c(t-s))\Big|ds d\tilde{ y}
 \end{align}
 that is \eqref{(*)} in case $n=1.$
 We now suppose \eqref{(*)} true up to $n,$ meaning
 $$  |q(y^1,y^1,\tilde x;t)|\leq (\frac{C_T}{2})^n\int_0^t\int_{\R^{d-1}}
  g_\alpha^{*n}(t-s_n) |q(y^1,y^1,\tilde y^n;s_n)|
  \phi_{d-1}(\tilde x-\tilde y_n;c(t-s_n))ds_nd\tilde y^n.
  $$
  Using \eqref{majqbis}  for  the factor 
  $q(y^1,y^1,\tilde y^n;s_n)$ yields

\begin{eqnarray*}
|q(y^1,y^1,\tilde x;t)|\leq (\frac{C_T}{2})^{n+1}  \int_0^t\int_{\R^{d-1}} g_\alpha^{*n}(t-s_{n})  \phi_{d-1}(\tilde x-\tilde y_{n};c(t-s_{n}))\\
\times\left[ \int_0^{s_n} \int_{{\mathbb R}^{d-1}} \Big| q( y^1,y^1,\tilde{y} ;s) g_\alpha (s_n-s)
\phi_{d-1}(\tilde y^n-\tilde y;c(s_n-s))\Big|ds d\tilde{ y}\right]
ds_{n}d\tilde y^{n}.
\end{eqnarray*}
  Using Tonelli's Theorem,  and properties of Gaussian densities convolution (see Remark \ref{phi1phi2} below), the integration in  $d\tilde y_n$ on $\R^{d-1}$ yields:
  
 $$\int_{\R^{d-1}}
  \phi_{d-1}(\tilde y^n -\tilde y;c(s_n-s))
  \phi_{d-1}(\tilde x-\tilde y_{n};c(t-s_{n}))d\tilde y_n=\phi_{d-1}(\tilde x-\tilde y;c(t-s))$$ 
  
 Thus, we get:
\begin{eqnarray*}
   |q(y^1,y^1,\tilde x;t)|\leq (\frac{C_T}{2})^{n+1}
   \int_0^t\int_{\R^{(d-1)}} \int_0^{s_n}
  g_\alpha^{*n}(t-s_{n})g_\alpha (s_n-s)
| q( y^1,y^1,\tilde{y} ;s) |\\
   \times \phi_{d-1}(\tilde x-\tilde y;c(t-s))ds_{n}ds d\tilde{ y}.
\end{eqnarray*}
Recall that $s \le s_n \le t$, we can rearrange the integrals to exhibit the convolution of $g_\alpha$ as the integral in $ds_n$, and since:
  $$
  \int_s^t
  g_\alpha^{*n}(t-s_{n})g_\alpha (s_n-s)ds_n= g_\alpha^{*(n+1)}(t-s).
  $$
  Consequently, we have established the property at rank $n+1$:
    \begin{eqnarray*}|q(y^1,y^1,\tilde x;t)|\leq (\frac{C_T}{2})^{n+1}
   \int_0^t\int_{\R^{(d-1)}} 
g_\alpha^{*(n+1)}(t-s)
| q( y^1,y^1,\tilde{y} ;s) |
\\ \times \phi_{d-1}(\tilde x-\tilde y;c(t-s))ds d\tilde{ y}.
  \end{eqnarray*}
 \end{proof}
 
 \begin{lemma}
 \label{lemmeLaure2}
 Let $M>0.$ Then
\begin{eqnarray*}
   \int_0^Tdt\int_{[-M,+M]^d}dy^1d\tilde x|q(y^1,y^1,\tilde x;t)|&\leq&
 (\frac{C_T}{2})^n(2M)^{d/2}T^{n\alpha/2+1}\frac{\Gamma^n(\alpha/2)}{\Gamma(n\alpha/2)}\\
 &&\times\int_0^Tds_n
  \sqrt{\int_{\R^d}|q(y^1,y^1,\tilde y^n;s_n)|^2dy^1d\tilde y}.
  \end{eqnarray*}
 \end{lemma}
 \begin{proof}
 According to Lemma \ref{lemme3Laure}:
 $$  |q(y^1,y^1,\tilde x;t)|\leq (\frac{C_T}{2})^n\int_0^t\int_{\R^{d-1}}
  g_\alpha^{*n}(t-s_n)|q(y^1,y^1,\tilde y^n;s_n)|
  \phi_{d-1}(\tilde x-\tilde y_n;c(t-s_n))ds_nd\tilde y^n,$$
  that we integrate on $[0,T]\times[-M,+M]^d$ so the following bound:
  
  \begin{eqnarray*}
   \int_0^T dt\int_{[-M,+M]^d}dy^1d\tilde x |q(y^1,y^1,\tilde x;t)| \leq
  (\frac{C_T}{2})^n\int_0^T dt\int_{[-M,+M]^d}dy^1d\tilde x \\ \times\int_0^t\int_{\R^{d-1}}
  g_\alpha^{*n}(t-s_n)|q(y^1,y^1,\tilde y^n;s_n)|
  \phi_{d-1}(\tilde x-\tilde y_n;c(t-s_n))ds_nd\tilde y^n.
  \end{eqnarray*}
  We change variables to  $\tilde x=\tilde y_n+\sqrt{t-s_n}\tilde z$ in the right hand side and use Tonelli's Theorem:
 
  \begin{eqnarray*}
 && \int_0^T dt\int_0^tds_n\int_{_{[-M,+M]^d\times\R^{d-1}}}dy^1d\tilde x d\tilde y^n 
  g_\alpha^{*n}(t-s_n)|q(y^1,y^1,\tilde y^n;s_n)|
  \phi_{d-1}(\tilde x-\tilde y^n;c(t-s_n)) \\ &=&
   \int_0^T dt\int_{\R^d}dy^1d\tilde z\1_{[-M,+M]}(y^1)\1_{[-M,+M]^{d-1}}(\tilde y_n+\sqrt{t-s_n}\tilde z)
  (\sqrt{t-s_n})^{d-1}\\ 
  &&\times
  \int_0^t\int_{\R^{d-1}}ds_nd\tilde y^n
  g_\alpha^{*n}(t-s_n)|q(y^1,y^1,\tilde y^n;s_n)|
  \phi_{d-1}(\sqrt{t-s_n}\tilde z;c(t-s_n)).
  \end{eqnarray*}
 
 Using Cauchy-Schwartz inequality, we bound the integral with respect to $dy^1d\tilde y^n:$ 
 \begin{eqnarray*}
 &&\int_{\R^{d}}dy^1d\tilde y^n|q(y^1,y^1,\tilde y^n;s_n)|
 \1_{[-M,+M]}(y_1)
  \1_{[-M,+M]^{d-1}}(\tilde y_n+\sqrt{t-s_n}\tilde z)\\ &\leq&\sqrt{\int_{\R^d}dy^1d\tilde y^n|q(y^1,y^1,\tilde y^n;s_n)|^2}\sqrt{\int_{\R^d}dy^1d\tilde y^n\1_{[-M,M]^d}(y^1,\tilde y^n+\sqrt{t-s_n}\tilde z)}\\
  &\leq& (2M)^{d/2}\sqrt{\int_{\R^d}dy^1d\tilde y^n|q(y^1,y^1,\tilde y^n;s_n)|^2}.
  \end{eqnarray*}
  We now remark that $\int_{\R^{d-1}} \phi_{d-1}(\sqrt{t-s_n}\tilde z;c(t-s_n))d\tilde z=1,$ and Lemma  11 Section \ref{technical:section} provides 
 for all $n\geq 0,$ $g_\alpha^{*n}(t)=t^{\frac{n\alpha}{2}-1}\frac{(\Gamma(\alpha/2))^n}{\Gamma(\frac{n\alpha}{2})}\leq T^{\frac{n\alpha}{2}-1}\frac{(\Gamma(\alpha/2))^n}{\Gamma(\frac{n\alpha}{2})}.$ We gather all these bounds
  \begin{eqnarray*}
   &&\int_0^T dt\int_{[-M,+M]^d}dy^1d\tilde x |q(y^1,y^1,\tilde x;t)|\\&\le&
  (\frac{C_T}{2})^n(2M)^{d/2}T^{n\alpha/2-1}\frac{(\Gamma(\alpha/2))^n}{\Gamma(\frac{n\alpha}{2})}\int_0^Tdt\int_0^tds_n
\sqrt{\int_{\R^d}dy^1d\tilde y^n|q(y^1,y^1,\tilde y^n;s_n)|^2}\\&\leq&
  (\frac{C_T}{2})^n(2M)^{d/2}T^{n\alpha/2+1}\frac{(\Gamma(\alpha/2))^n}{\Gamma(\frac{n\alpha}{2})}
\sqrt{\int_{\R^d}dy^1d\tilde y^n|q(y^1,y^1,\tilde y^n;s_n)|^2}
  \end{eqnarray*}
  which conclude the proof.
\end{proof}
We now can prove Theorem 2: according to Item (b) 
$$\int_0^Tds_n
  \sqrt{\int_{\R^d}|q(y^1,y^1,\tilde y^n;s_n)|^2dy^1d\tilde y}<\infty.
$$
In Lemma \ref{lemmeLaure2} the bound 
$ (\frac{C_T}{2})^n(2M)^{d/2}T^{n\alpha/2+1}\frac{\Gamma^n(\alpha/2)}{\Gamma(n\alpha/2)}$ goes to $0$ when $n$ goes to infinity (Stirling formula), so $\forall M>0$
$$ \int_0^Tdt\int_{[-M,+M]}dy^1d\tilde x|q(y^1,y^1,\tilde x;t)|=0
$$
meaning that $dtdy^1d\tilde x$ almost surely $q(y^1,y^1,\tilde x;t)=0$ on 
$[0,T]\times [-M,+M]^d$ for all $M.$ 
Thus $q(y^1,y^1,\tilde x;t)=0$   $dtdy^1d\tilde x$ almost surely. This achieves the proof of Theorem \ref{unique0}.
\\

\section{Checking assumptions of Theorem \ref{ball} for $q_H$}
\label{checkingBall}
In this section, we check that the assumptions of Ball's Theorem \ref{ball} actually holds for $q_H$ in equation \eqref{edp-qH}.
The Banach space $X$ is $L^2(\R^d).$
{
We have to prove that

\begin{itemize}
\item[-]   ${\cal L^*}$ is a densely defined closed linear operator on $X$,
and  ${\cal L^*}$ generates a strongly continuous { on $X$} semi-group $(Q_t)_{t \le T}$ { such that the operators $Q_t$ are bounded on $L^2(\R^d)$}:
this is done in Proposition \ref{lem-semi-groupe-dual}.

\item[-]    $f$ belongs to $L^1([0,T], X)$
where $f(t,x)= \frac12 H'(x^1) q(x^1,x^1,\tilde{x};t) .$

\item[-]   $u\in C([0,T],X)$ where
$ u = q_H.$ 
\end{itemize}

 \subsection{The function 
 $f$ belongs to $L^1([0,T], X)$}

Recall that this function is defined as
$$
f:t\mapsto  \left(x\mapsto \frac12  H'(x^1) q(x^1,x^1,\tilde{x};t) \right).
$$
Since $H'$ is bounded, this is a consequence of Item (b) { satisfied by $q \in {\cal X}.$}

\subsection{The function $u=q_H$ belongs to $\mathcal{C}([0,T],X),$  where $X=L^2(\R^d)$.}

\begin{lemma}
\label{contqH}
{ Assume  $q$ satisfies Items (a) and (b), meaning}
\begin{eqnarray}
\label{HYPball}
&&\sup_{{ t\in]0,T]}}\int_{\R^d}\left[\int_{m\geq x^1} |q(m,x;t)|dm\right]^2dx<\infty
\\
&&
\label{HYPball2}
{\int_0^t ds\sqrt{\int_{\R^d} |q(x^1,x;s)|^2 dx} <\infty,}
\end{eqnarray}
then $q_H(x;t):=\int_\R H(m)q(m,x;t)dm$
satisfies $q_H\in C\left([0,T],L^2(\R^d)\right).$
\end{lemma}
%
 \begin{proof}
{\bf Step 1}: We show that $\sup_{t\in[0,T]} \|q_H(.;t)\|_{L^2(\R^d)} <+\infty.$
\\
Remark that for any $t\in [0,T]$ and recalling $H\in C^1_b(\R)$:
$$\int_{\R^d}q_H^2(x;t)dx=
\int_{m\geq x^1} \left(\int_{\R^d} H(m)q(m,x;t)dm\right)^2dx\leq\|H\|_\infty^2 \int_{\R^d}\left(\int_{m\geq x^1}  |q(m,x;t)|dm\right)^2dx
$$
and using $\sup_{t\in[0,T]}$ on both sides, Step 1 is proved according to Assumption \eqref{HYPball} which is Item (a).
\\
\\
{\bf Step 2}:  for all $\varphi\in L^2(\R^d), ~t\mapsto \int_{\R^d}\varphi(x)q_H(x;t)dx$ is  continuous.

 $\bullet$ First we consider 
 { $\varphi \in C^2_c(\R^d)$ and  $\psi(m,x)=H(m)\varphi(x).$} According to the PDE \eqref{edp1m} (noting that
$q=p^1-p^2$ so the first term is zero)
\begin{eqnarray*}
&&\int_{\R^d}\varphi(x)q_H(x,;)dx=
\int_{\cal T}\varphi(x) H(m)q(m,x;t)dmdx=
\\
&&\int_0^t\int_{\cal T} q(m,x;s)H(m){\cal L}\varphi(x)dmdxds
+{ \frac{1}{2}}
\int_0^t\int_{\R^d} H'(m)\varphi(m,\tilde x ){  q(m,m,\tilde x;s)}dmd\tilde xds
\end{eqnarray*}
where we identify (using Lebesgue Theorem)
\begin{eqnarray*}
&&\int_{\R^d}\varphi(x)q_H(x;t)dx=\\
&&\int_0^t
\int_{\R^d}
 q_H(x;u)
 {\cal L}\varphi(x)dxdu
+ \frac{1}{2}
\int_0^t\int_{\R^d} H'(m)\varphi(m,\tilde x ){  q(m,m,\tilde x;u)}dmd\tilde xdu.
\end{eqnarray*}
Using Cauchy-Schwartz, boundness of $H'$
and denoting $\tilde q(x;u)=q(x^1,x^1,\tilde x;u)$:
\begin{eqnarray}\label{eq-cont-q_h}
&&\left|\int_{\R^d}\varphi(x)[q_H(x;t)-q_H(x;s)]dx\right|\leq
\\
&&\int_s^t
 \|q_H(.;u)\|_{L^2(\R^d)}
 \|{\cal L}\varphi\|_{L^2(\R^d)}du
+ \frac{1}{2}
\int_s^t\|H'\|_\infty \|\varphi\|_{L^2(\R^d)} 
 \|\tilde q(.;u)\|_{L^2(\R^d)}du.\nonumber
\end{eqnarray}
When $f \in L^1[0,T]$, the function $t \mapsto \int_0^t f(s) ds$ is continuous. According to Step 1 and Assumption \eqref{HYPball2} on $\tilde{q}$, we deduce that the function:
$$ t \mapsto \int_0^t
 \|q_H(.;u)\|_{L^2(\R^d)}
 \|{\cal L}\varphi\|_{L^2(\R^d)}du
+ \frac{1}{2}
\int_0^t\|H'\|_\infty \|\varphi\|_{L^2(\R^d)} 
 \|\tilde q(.;u)\|_{L^2(\R^d)}du$$
 is continuous.
 Thus, for all  $\varepsilon >0,$ there exists $\eta >0,$ such that for $|t-s|<\eta$ 
 \begin{align*}
 \int_s^t
 \|q_H(.;u)\|_{L^2(\R^d)}
 \|{\cal L}\varphi\|_{L^2(\R^d)}du
+ \frac{1}{2}
\int_0^t\|H'\|_\infty \|\varphi\|_{L^2(\R^d)} 
 \|\tilde q(.;u)\|_{L^2(\R^d)}du\leq \varepsilon,
 \end{align*}
and according to \eqref{eq-cont-q_h}
 \begin{align*}
 \left|\int_{\R^d}\varphi(x)[q_H(x;t)-q_H(x;s)]dx\right|<\varepsilon
 \end{align*}
 so yields the  continuity of the map
$t\mapsto \int_{\R^d}\varphi(x)q_H(x;t)dx
$ for $\varphi \in C^2_c(\R^d)$.
 

$\bullet$ Secondly, we prove the
 continuity of the map
$t\mapsto \int_{\R^d}\varphi(x)q_H(x;t)dx
$ for $\varphi \in L^2(\R^d)$.
\\
There exists for all $\varepsilon>0$ a function $\varphi_\varepsilon \in C^2_c(\R^d)$  such that
$$
\|\varphi-\varphi_\varepsilon\|_{L^2(\R^d)}\leq\varepsilon.
$$
Let $0<s<t$:
\begin{eqnarray*}
\int_{\R^d}\varphi(x)(q_H(x;t)-q_H(x;s))dx&=&
\int_{\R^d}(\varphi-\varphi_\varepsilon )(x)q_H(x;t)
\\
&&-\int_{\R^d}(\varphi-\varphi_\varepsilon)(x)q_H(x;s)dx
+\int_{\R^d}\varphi_\varepsilon(x)[q_H(x;t)-q_H(x;s)]dx.
\end{eqnarray*}
Using the continuity of ${ t\mapsto}\int_{\R^d}\varphi_\varepsilon(x)q_H(x;t)dx$ { due to} $\varphi_\varepsilon\in C^2_c$, there exists $\eta$ such that $t-s\leq \eta$ yields
$$|\int_{\R^d}\varphi_\varepsilon(x)[q_H(x;t)-q_H(x;s)]dx|\leq \varepsilon.
$$
We now use the following for $u=t$ and $s$ $\in [0,T]$
$$|\int_{\R^d}(\varphi-\varphi_\varepsilon)(x)q_H(x;u)dx|
\leq 
\|\varphi-\varphi_\varepsilon\|_2
\sup_{u\in[0,T]}[\int_{\R^d}(q_H(x;u))^2dx]^\frac12.$$
Thus for $0\leq s\leq t\leq T,~t-s<\eta$ gathering these three bounds
$$\left|\int_{\R^d}\varphi(x)(q_H(x;t)-q_H(x;s))dx\right|\leq \varepsilon(1+2\sup_{u\in[0,T]}[\int_{\R^d}(q_H(x;u))^2dx]^\frac12)
$$
which concludes the uniform continuity of the map
$t\mapsto \int_{\R^d}\varphi(x)q_H(x;t)dx
$ for $\varphi \in L^2(\R^d)$.
\\
\\
{\bf Step 3}: $q_H\in C\left([0,T],L^2(\R^d)\right).$
\\
The previous step gives the weak continuity, the first step the uniform bound in $L^2,$ so
the continuity in $C([0,T],L^2(\R^d))$ is proved.

\end{proof}


\section{The density $p_V$ is an element of space  ${\mathcal X}$}
 \label{pVdansX}

 Let us denote $p_V(\cdot;t,x_0)$ 
the density of law of $(M_t,X_t)$ when the initial condition is $X_0=x_0.$ { We recall that
$f_0 $ is the square integrable   density of $ X_0.$}
 Thus
$$p_V(m,x;t)=\int_{\R^d}p_V(m,x;t,x_0)f_0(x_0)dx_0.
$$
The aim in this section is to check  that the function $p_V$ defined above satisfies Items (a)  (b) and (c) in Definition~\ref{def-espace-sol}}.

\begin{prop}
\label{itema}
Item (a) is satisfied by $p_V$:
\begin{eqnarray*}
&&\sup_{t\in]0,T]}\int_{\R^d}\left[\int_\R |p_V(m,x;t)|dm\right]^2dx<\infty.
\end{eqnarray*}
\end{prop}

\begin{proof}
According to \cite{coutin:pontier:2020} (page 26 (ii)) there exists $C_T$ such that
\begin{equation}
\label{pVversusGauss}
p_V(m,x;t,x_0)\leq C_T\phi_{d+1}(m-x^1_0,m-x^1,\tilde x-\tilde x_0;{ 2t}){\mathbf 1}_{\{m>\max(x^1,x_0^1)\}}
\end{equation}
where $\phi_k(u^1, \dots,u^k;t)=\frac{1}{(\sqrt{2\pi t})^k}
e^{-\frac{1}{2t}\sum_{i=1}^k (u^i)^2},$ $u=(u^1,\dots,u^k)\in \R^k$.

\begin{remark} ~~
\label{phi1phi2}

\begin{itemize}
\item  Consider two $\R^d$ valued independent Gaussian variables
$X\sim {\cal N}(a,s)$ and $Y\sim {\cal N}(b,r)$,
 the sum
$-X+Y$ has distribution ${\cal N}(b-a,s+r)$ and the Gaussian densities satisfy the convolution property:
$${ \int_{\R^d}}\phi_d(x-a;s)\phi_d(x-b;r)dx= \phi_d(b-a;r+s).
$$
\item Besides, we can split the $k$ dimensional density $\phi_k(u^1, \dots, u^k;t)$ as a product of its marginals $\phi_1(u^i;t)$:
$$
\phi_k(u^1, \dots,u^k;t) = \phi_1(u^1;t) \times \cdots \times \phi_1(u^k;t).
$$
\end{itemize}
\end{remark}
%
 According to \eqref{pVversusGauss}
\begin{eqnarray}
A_t&:=&\int_{\R^d}
\left|\int_\R p_V(m,x;t)dm\right|^2dx\nonumber
\\
&\leq& C_T^2\int_{\R^{d}}\left|\int_\R\int_{\R^d}\phi_{d+1}(m-x^1_0,m-x^1,\tilde x-\tilde x_0;2t)f_0(x_0)dx_0dm\right|^2dx.
\end{eqnarray}
Integrating with respect to $m$ and using Remark \ref{phi1phi2} with $d=1$ and
\\
$\phi_{d+1}(m-x^1_0,m-x^1,\tilde x-\tilde x_0;2t)=\phi_1(m-x^1_0;2t)\phi_1(m-x^1;2t)\phi_{d-1}(\tilde x-\tilde x_0;2t)$, we get:
$$\int_\R\phi_{d+1}(m-x^1_0,m-x^1,\tilde x-\tilde x_0;2t)dm=\phi_1(x^1-x^1_0;4t)\phi_{d-1}(\tilde x-\tilde x_0;2t).
$$
Thus, we have :
\begin{eqnarray*}
\int_{\R^d}
\left|\int_\R p_V(m,x;t)dm\right|^2dx
&\leq&  C_T^2\int_{\R^{3d}}
\phi_1(x^1-x^1_0;4t)\phi_1(x^1-y^1_0;4t)
\\
&&\times\phi_{d-1}(\tilde x-\tilde x_0;2t)\phi_{d-1}(\tilde x-\tilde y_0;2t)
f_0(x_0)f_0(y_0)d x dx_0dy_0.
\end{eqnarray*}
Then integrating with respect to $x^1,\tilde x$ and using once again Remark  \ref{phi1phi2}:
\begin{eqnarray*}
&&\int_{\R^d}
\left|\int_\R p_V(m,x;t)dm\right|^2
dx
\\
&\leq&  C_T^2\int_{\R^{2d}}
\phi_1(x_0^1-y^1_0;{ 8t})
\phi_{d-1}(\tilde x_0-\tilde y_0;4t)
f_0(x_0)f_0(y_0) dx_0dy_0\\
&\le&  C_T^2\int_{\R^{2d}}
\phi_d( x_0- y_0;8t)
f_0(x_0)f_0(y_0) dx_0dy_0,
\end{eqnarray*}
where we have used that  $\max(\phi_k(y;t), \phi_k(y;2t) )\leq (\sqrt 2)^k\phi_k(y;2t)$.
Let $z=x_0-y_0$  so $y_0=x_0 -z$,  with respect to a multiplicative constant 
$$ \int_{\R^{2d}}
\phi_1(x_0^1-y^1_0;{ 8t})
\phi_{d-1}(\tilde x_0-\tilde y_0;4t)
f_0(x_0)f_0(y_0) dx_0dy_0\leq 
\int_{{\mathbb R}^{2d}} \phi_d( z;{ 8t}) f_0(x_0) f_0(x_0-z) dz dx_0.$$

We first integrate with respect to  $x_0$ so, the  Cauchy-Schwartz inequality yields the upper bound:
$$ \int_{{\mathbb R}^{d} }\phi_d( z;{ 8t}) \left(\sqrt{\int_{\R}f_0^2(x_0) dx_0 \int_{\R}  f_0^2(x_0-z) dx_0}\right) dz$$
then using the change of variable $u= x_0-z$
in the last factor under the square root, since 
$$ \int_{{\mathbb R}^{d} }\phi_d( z;{ 8t})dz=1,$$ the upper bound is
$$A_t\leq C_T
 \sqrt{\int_{\R}f_0^2(x_0) dx_0 \int_{\R}  f_0^2(u)  du}=C_T\|f_0\|_{L^2}^2 ,$$
meaning 
$A_t=\int_{\R^d}
\left|\int_\R p_V(m,x;t)dm\right|^2dx<\infty$ according to assumption on $f_0$ in Theorem \ref{mthm}.
\end{proof}

\begin{prop}
\label{itemb}
Item  (b) is satisfied by $p_V$:
\begin{eqnarray*}
{ \int_0^T ds\sqrt{\int_{\R^d} |p_V(x^1,x;s)|^2 dx} <\infty.}
\end{eqnarray*}
\end{prop}

\begin{proof} 
Using \eqref{pVversusGauss} with $m=x^1$
\begin{eqnarray*}
p_V(x^1,x;t)\leq C_T\int_{\R^d}\phi_{d+1}(x^1-x^1_0,0,\tilde x-\tilde x_0;2t)f_0(x_0)dx_0
\\
=\frac{C_T}{\sqrt{4\pi t}}\int_{\R^d}
\phi_{d}(x^1-x^1_0,\tilde x-\tilde x_0;2t)f_0(x_0)dx_0.
\end{eqnarray*}
We compute  the $L^2$ norm
\begin{eqnarray*}
&\int_{\R^d}(p_V(x^1,x;t))^2dx\\
&~~~~\leq \frac{C_T^2}{4\pi t}
\int_{\R^{3d}}
\phi_{d}(x^1-x^1_0,\tilde x-\tilde x_0;2t)\phi_{d}(x^1-y^1_0,\tilde x-\tilde y_0;2t)f_0(x_0)f_0(y_0)dx_0dy_0dx.
\end{eqnarray*}
We integrate with respect to  $dx$
and use Remark \ref{phi1phi2}:
\begin{eqnarray*}
\int_{\R^d}(p_V(x^1,x;t))^2dx
\leq \frac{C_T^2}{4\pi t}
\int_{\R^{2d}}
\phi_{d}(x^1_0-y^1_0,\tilde x_0-\tilde y_0;4t)f_0(x_0)f_0(y_0)dx_0dy_0.
\end{eqnarray*}
We operate the change of variable $z=x_0-y_0$
\begin{eqnarray*}
\int_{\R^d}(p_V(x^1,x;t))^2dx
\leq \frac{C_T^2}{4\pi t}
\int_{\R^{2d}}
\phi_{d}( z;4t)f_0(x_0)f(x_0-z)dx_0dz.
\end{eqnarray*}
Using Cauchy-Schwartz inegality for the measure $dx_0$:
\begin{eqnarray*}
\int_{\R^d}(p_V(x^1,x;t))^2dx
\leq \frac{C_T^2}{4\pi t}
\int_{\R^{d}}
\phi_{d}( z;4t)\|f_0\|_{L^2(\R^d)}^2dz=
\frac{C_T^2}{4\pi t}\|f_0\|_{L^2(\R^d)}^2.
\end{eqnarray*}
{
The integrability in time of this $L^2(\R^d)$-norm
yields Item (b):
$$\int_0^Tdt\sqrt{\int_{\R^d}(p_V(x^1,x;t))^2dx}
\leq \int_0^Tdt\frac{C_T}{\sqrt{2\pi t}}\|f_0\|_{L^2(\R^d)}<\infty.
$$
}
\end{proof}

We now prove that $p_V$ satisfies Item (c) in Definition \ref{def-espace-sol}.
\\
Note that Item (c) is actually stronger than Hypothesis 2.1 in \cite{coutin:pontier:2020}, which is satisfied in the case where $A = I_d$ (our current setting) and $d = 1$. For clarity, we restate the Hypothesis 2.1 of \cite{coutin:pontier:2020} here:

For all $t > 0$, the density of the distribution of the random vector $(M_t, X_t)$ with respect to the Lebesgue measure, denoted by $p_V$, satisfies:

\begin{itemize} \item[(i)] $(t, m, \tilde{x}) \mapsto \sup_{u > 0} p_V(m, m - u, \tilde{x}; t)$ belongs to $L^1\Big([0, T] \times \mathbb{R}^d, dt  dm  d\tilde{x}\Big).$ 

\item[(ii)] For each $t$, almost surely in $(m, \tilde{x}) \in \mathbb{R}^d$, the limit $\lim_{u \rightarrow 0^+} p_V(m, m - u, \tilde{x}; t)$ exists and is denoted by $p_V(m, m, \tilde{x}; t).$ \end{itemize}

However, Property (ii) alone is insufficient to derive Item (c), as we require continuity of $x^1 \mapsto p_V(m, x^1, \tilde{x}; t)$ throughout ${\cal T}$, not just on the boundary.

\begin{prop}[Item (c) is satisfied by $p_V$]
\label{itemc}
 For all $t\in ]0,T],$ 
 $(m,\tilde x)\mapsto\sup_{u>0}p(m,m-u,\tilde x;t)\in L^1(\R^d);$ 
 and for  almost surely in $(m,\tilde{x}) \in {\mathbb R}^d,$ 
 $x^1 \mapsto p(m,x^1,\tilde{x};t)$ is continuous on $]-\infty, m[$  and  $\lim_{u\rightarrow 0^+} p(m,m-u,\tilde{x};t)$
 exists and is denoted by $ p(m,m,\tilde{x};t).$
%
\end{prop}

The proof is a consequence of the following lemmas and propositions.
First, we prove the integrability assumption.

\begin{lemma}
\label{item-c-integ}
For all $t\in ]0,T],$  $(m,\tilde x)\mapsto\sup_{u>0}p_V(m,m-u,\tilde x;t)\in L^1(\R^d)$  .
\end{lemma} 

\begin{proof} 
Recall that 
$
p_V(m,x;t)= \int_{{\mathbb R}^d} p_V(m,x,t;x_0)f_0(x_0) dx_0
$
where $p_V(m,x;t,x_0)$ is the density of the law of $(M_t,X_t)$ when the initial condition is $X_0=x_0.$
Moreover, { recall \eqref{pVversusGauss}}
\begin{align*}
| p_V(m,x;t, x_0)| \leq C_T \phi_{d+1}(m-x_0^1,m-x^1,\tilde{x}- \tilde{x_0};2t)
 {\mathbf 1}_{ m > \max(x^1,x_0^1)}.
\end{align*}
Then, using the fact that $e^{-\frac{u^2}{4t}}\leq 1$ we obtain { $\forall t\in ]0,T]$}
\begin{align*}
\sup_{u>0}p_V(m,m-u,\tilde x;t)\leq \frac{C_T}{\sqrt{4 \pi t}} \int_{\R^d} \phi_{d}(m-x_0^1,\tilde{x}- \tilde{x_0};2t)f_0(x_0)dx_0.
\end{align*}
Integrating with respect to $m$ and $\tilde{x}$  we obtain { $\forall t\in ]0,T]$:}
\begin{align*}
\int_{\R^d} \sup_{u>0}p_V(m,m-u,\tilde x;t){ dmd\tilde x}\leq \frac{C_T}{\sqrt{4 \pi t}} \int_{\R^d} f_0(x_0)dx_0=\frac{C_T}{\sqrt{4 \pi t}}<\infty.
\end{align*}

\end{proof}

The following Lemma is Proposition 4.5 of \cite{coutin:pontier:2020}.
\begin{lemma}\label{lem-c-cont-m}
For  all  $(t,m,\tilde{x}) \in ]0,T] \times{\mathbb R}^d,$ 
 $\lim_{u\rightarrow 0^+} p_V(m,m-u,\tilde{x};t)$
 exists 
 and is denoted by $ p_V(m,m,\tilde{x};t).$
\end{lemma}
Now, we turn to the continuity of $p_V$  for $x^1 \in ]-\infty, m[.$
\begin{prop}\label{pro-cont-x-1}
 For all
$(t,m,\tilde{x}) \in ]0,T] \times {\mathbb R}^d,$ 
{ $x^1 \mapsto p_V(m,x^1,\tilde{x};t)$ is continuous on $]-\infty, m[.$ }
\end{prop}
This proposition extends the results of Proposition 4.5 in \cite{coutin:pontier:2020}. Its proof follows almost the same lines and is split into three lemmas.
We start by recalling Proposition 4.2 of~\cite{coutin:pontier:2020}:
\begin{lemma}
\label{rem-dens-mal} 
On $[0,T]\times {\mathbb R}^{d+1}$
\begin{align}
\label{rep-p_v}
&p_V=p_0-\sum_{k=1}^{d+1}\left( { p^{k,\alpha}} +{ p^{k,\beta}}\right)
\end{align}
where { the various $p$ are defined as}
\begin{align*}
 &p_0(m,x;t):={ \int_{{\mathbb R}^d}p_{W^{*1},W}(m-x_0^1,x-x_0;t)f_0(x_0)dx_0},
  \\
& p^{k,\alpha}(m,x;t):=\int_0^t\int_{\R^{d+1}} 
{\mathbf 1}_{m>b}
B^k(a)\partial_{x^k}p_{W^{*1},W}(m-a^1,x-a;t-s)
p_V(b,a;s)dbdads,
\\
& p^{k,\beta}(m,x;t)
:=\int_0^t \int_{\mathbb R^{d+1}}
{\mathbf 1}_{m>b}
 B^k(a)\partial_{x^k}p_{W^{*1},W}(b-a^1,x-a;t-s)p_V(m,a;s)dbda ds
\end{align*}
{ where $\partial_{x^k}$ is the derivative with respect to  $k=m, x^1,\dots,x^d,$
and $B^m=B^1$.}
\end{lemma}
Using the same lines as the proof of Lemma 4.5 of \cite{coutin:pontier:2020}, one can prove the following Lemma:
\begin{lemma}
For all $(t,m,\tilde{x}),$ $x^1 \mapsto p_0(m,x^1,\tilde{x};t)$ is continuous on $]-\infty,m[.$
\end{lemma}
Using the same lines as the proof of Lemma 4.6 of \cite{coutin:pontier:2020}, one can prove the following Lemma:
\begin{lemma}
For all $(t,m,\tilde{x}),$ $x^1 \mapsto p^{k,\alpha}(m,x^1,\tilde{x};t)$ is continuous on $]-\infty,m[$
for $k=m,1,...,d.$
\end{lemma}
Unfortunately, for the continuity of the maps $p^{k,\beta}$, we cannot use the same arguments as the proof of Lemma 4.7 of \cite{coutin:pontier:2020},{ but we can prove the following.}

\begin{lemma}
\label{lem-cont-beta}
For all $(t,m,\tilde{x}),$ $x^1 \mapsto p^{k,\beta}(m,x^1,\tilde{x};t)$ is continuous on $]-\infty,m[$
for $k=m,1,...,d.$
\end{lemma}
 \begin{proof} Instead of using explicit computations as    the proof of Lemma 4.7  of \cite{coutin:pontier:2020}, we use Vitali Convergence theorem.
Let us introduce the measure $\nu$ on $[0,t] \times \R^{2d+1}$ defined by
\begin{align*}
\nu(dsdbdadx_0) = {\mathbf 1}_{]0,t]}(s) {\mathbf 1}_{a^1 <b<m} p_V(m,a,s;x_0) f_0(x_0)dsdadbdx_0.
\end{align*}

(i) We show that $\nu$ is a finite measure on $[0,t] \times \R^{2d+1}$.
Recall that from the upper bound \eqref{pVversusGauss} on $p_V$, we get:
\begin{align*}
\nu( [0,t] \times \R^{2d+1}) &\leq  C_T\int_0^t \int_{\R^{2d+1}}{\mathbf 1}_{a^1<b<m}\phi_{d+1}(m-x_0^1,m-a^1,\tilde{a}- \tilde{x_0};2s)f_0(x_0) 
 ds dadb dx_0
 \\
 &\leq  C_T \int_0^t\int_{\R^{2d}}(m-a^1)_{{ +}}\phi_{d+1}(m-x_0^1,m-a^1,\tilde{a}- \tilde{x_0};2s)f_0(x_0) 
 ds dadx_0
\end{align*}
after integrating with respect to $b.$ Integrating with respect to $\tilde{a}$:
\begin{align*}
\nu( [0,t] \times \R^{2d+1}) &\leq  C_T\int_0^t  \int_{\R^{2}}(m-a^1)_{{ +}}\phi_{2}(m-x_0^1,m-a^1;2s)f_0(x_0) 
 ds da^1dx_0.
\end{align*}
Then integrating with respect to $a^1$ and using  $\sqrt{\frac{s}{\pi}}\phi_{1}(m-x_0^1;2s)\leq 1/2\pi$:
\begin{align}\label{maj-nu-finie}
\nu( [0,t] \times \R^{2d+1}) &\leq  C_T\int_0^t  \int_{\R}{\sqrt{\frac{s}{\pi}}}\phi_{1}(m-x_0^1;2s)f_0(x_0) 
 ds dx_0\nonumber
 \\
 &\leq \frac{C_T}{2\pi}\int_0^t \int_{\R} f_0(x_0) dx_0ds\leq  \frac{TC_T}{2\pi}<\infty
\end{align}

That means that $\nu$ is a finite measure on $[0,t] \times \R^{2d+1}.$
\\
(ii) We can now prove the continuity on 
$(-\infty,m)$ of the application $x^1\mapsto p^{k,\beta}(m,x^1,\tilde{x};t)$.
 
Note that 
\begin{align*}
 p^{k,\beta}(m,x^1,\tilde{x};t)=\int_{[0,T] \times \R^{2d+1}}{ \1_{b<m}}B^k(a) \partial_{x^k} p_{W^{*1},W}(b-a^1,x-a;t-s)  \nu(dsdbdadx_0).
 \end{align*}
 According to identity (57) of \cite{coutin:pontier:2020}
 \begin{align*}
 p_{W^{*1},W}(b-a^1,x^1-a^1;t)=-2 \partial_{x^1} \phi_d(2b-a^1-x^1,\tilde{x}-\tilde{a};t){\mathbf 1}_{b \geq \max (a^1,x^1)}.
 \end{align*}
 Thus, $x^1 \mapsto  \partial_{x^k} p_{W^{*1},W}(b-a^1,x^1-a^1,\tilde{x}-\tilde{a};t)$ is continuous for { $x^1< b.$}
\\
Let
$$I(x^1):=\int_{[0,t] \times {\mathbb R}^{2d+1}}|B^k(a)  \partial_{x^k} p_{W^{*1},W}(b-a^1,x^1-a^1,\tilde{x}-\tilde{a};t-s)|^{\varepsilon} \nu(dsdbdadx_0).
$$
Assume for a while that 
 \begin{align}
 \label{maj-ui}
\sup_{x^1 <m}I(x^1)  <\infty
\end{align}
holds for $1<\varepsilon < \frac{d+3}{d+2}$.
In this case, using Vitali convergence Theorem, we have that for all $(t,m,\tilde{x}),$ the function $x^1 \mapsto p^{k,\beta}(m,x^1,\tilde{x};t)$ is continuous on $]-\infty,m[$
for $k=m,1,...,d.$
\\
(iii) It remains to prove the estimate \eqref{maj-ui}.
Recall that from the upper bound (55) of Lemma A.2 of \cite{coutin:pontier:2020}, there exists a constant $D$ such that for $x=b,a^1,...,a^d$
\begin{align*}
|\partial_x p_{W^{*1},W}(b,a;t) | \leq \frac{D}{\sqrt{t}}\phi_{d+1}(b,b-a^1,\tilde{a}:2t){\mathbf 1}_{b>\max(a^1,0)}.
\end{align*}
Let $\gamma= \varepsilon  +(\varepsilon-1)(d+1) < 2.$  Since
  $\varepsilon>1,$  $[\phi_{d+1}(.;t)]^\varepsilon \leq [\sqrt{2\pi t}]^{(1-\varepsilon)(d+1)}\phi_{d+1}(.;t)$ 
  and $B$ { is bounded so: }
{\begin{align*}
|B^k(a) \partial_x p_{W^{*1},W}(b-a^1,x^1-a^1,\tilde{x}- \tilde{a};t-s) |^{\varepsilon} \leq \frac{D^{\varepsilon}\|B\|_{\infty}^{\varepsilon}}{\sqrt{(t-s)^{\gamma}}}\phi_{d+1}(b-a^1,b-x^1,\tilde{a}:2(t-s)),
\end{align*}
}
Then,  up to a multiplicative constant,  using once again \eqref{pVversusGauss}  to bound the density of the measure~$\nu$, we get:
\begin{align*}
I(x^1) \leq 
\int_{[0,t] \times {\mathbb R}^{2d+1}}&
\frac{D^{\varepsilon}\|B\|_{\infty}^{\varepsilon}}{\sqrt{(t-s)^{\gamma}}}\phi_{d+1}(b-a^1,b-x^1,\tilde{a}-\tilde{x};2(t-s))
\\
&
{ C_T}\phi_{d+1}(m-x_0^1,m-a^1,\tilde{a}- \tilde{x_0};2s)f_0(x_0) 
 ds dadb dx_0.
 \end{align*}
Recall Remark \ref{phi1phi2}
$
\int_{\R^d} \phi_d(x-a;s)\phi_d(x-b;r)dx= \phi_d(b-a;s+r)
$ that we will use  several times.
We integrate with respect to $a$ so
 \begin{align*}
I(x^1) \leq 
\int_{[0,t] \times {\mathbb R}^{d+2}}&
\frac{D^{\varepsilon}\|B\|_{\infty}^{\varepsilon}}{\sqrt{(t-s)^{\gamma}}}\phi_{1}(b-x^1;2(t-s))
\\&
\phi_{1}(m-x_0^1;2s)\phi_{d-1}(\tilde{x}-\tilde{x_0};4t)\phi_1(b-m;4t)f_0(x_0) 
 ds db dx_0,
 \end{align*}
then, we integrate with respect to $b:$
  \begin{align*}
I(x^1) \leq 
\int_{[0,t] \times {\mathbb R}^{d+2}}&
\frac{D^{\varepsilon}\|B\|_{\infty}^{\varepsilon}}{\sqrt{(t-s)^{\gamma}}}\phi_{1}(m-x^1;2(3t-s))
\\&
\phi_{1}(m-x_0^1;2s)\phi_{d-1}(\tilde{x}-\tilde{x_0};2t)f_0(x_0) 
 ds  dx_0.
 \end{align*}
 Since $s<t$ then $3t-s >t$  then
 $$\phi_{1}(m-x^1;2(3t-s))\phi_{1}(m-x_0^1;2s) \phi_{d-1}(\tilde{x}-\tilde{x_0};2t) \leq [\sqrt{4\pi t}]^{-d}[\sqrt{4 \pi s}]^{-1}$$ and
 \begin{align*}
 I(x^1) \leq \int_0^t \int_{R^d} \frac{D^{\varepsilon}\|B\|_{\infty}^{\varepsilon}}{\sqrt{(4\pi t)^d}\sqrt{(t-s)^{\gamma}}\sqrt{4\pi s}}f_0(x_0) dx_0 ds.
 \end{align*}
 Since $\int_{R^d} f_0(x_0)dx_0=1,$ $\forall x^1<m$:
 \begin{align*}
 I(x^1) \leq \int_0^t  \frac{D^{\varepsilon}\|B\|_{\infty}^{\varepsilon}}{\sqrt{(4\pi t)^d}\sqrt{(t-s)^{\gamma}}\sqrt{4 \pi s}} ds  
 \end{align*}
 { is finite since  $\gamma <2$ 
 and so $\sup_{x^1<m} I(x^1)<\infty,$ 
and  the proof of \eqref{maj-ui} is achieved, and so is the one of Lemma \ref{lem-cont-beta}.}
 \end{proof}   
 
 The proof of Proposition \ref{pro-cont-x-1} is achieved using Lemmas \ref{rem-dens-mal} to \ref{lem-cont-beta}.

 The proof of Proposition \ref{itemc} is then a consequence of Lemmas \ref{item-c-integ} and \ref{lem-c-cont-m} and Proposition~\ref{pro-cont-x-1}.

{ The following Proposition concludes this section:
\begin{prop}
The density of probability $p_V$, density of the law of $(M_t,X_t)$ where
\begin{eqnarray*}
X_t&=&X_0+\int_0^tB(X_s)ds+W_t,\\
M_t&=& \sup_{s \le t} X_t^1,
\end{eqnarray*} 
with $X_0$ independent of $W$ with density
$f_0 \in L^1({\mathbb R}^d)\cap L^2({\mathbb R}^d)$,
is an element of the set ${\cal X}$. 
\end{prop}
 This is a consequence of Propositions 
\ref{itema}, \ref{itemb}, \ref{itemc}, and Theorem \ref{mthm} is proved.}

\section{Tools}
\label{technical:section}

In this section, we recall and prove some technical results needed throughout this paper.

\begin{prop}
\label{lem-semi-groupe-dual}
The adjoint operator 
$${\cal L}^*f (x) =  -B(x) \cdot \nabla f(x)-f(x){ \sum_{k=1}^d\partial_{x^k} B^k(x)}+{\frac12}\Delta f(x), $$
 is closed and generates a semi-group $(Q_t)_{t \ge 0}$ that is strongly continuous and bounded
 { with kernel $\Gamma$} on $L^2({\mathbb R}).$ 
 Besides, for $0 \le t \le T$, $x\in \R^d$ and $f\in L^2(\R^d)$ , 
 \begin{align}
 \label{defQ}
 Q_t(f)(x):=E\left[f(Y^x_t)\exp\left(-\int_0^t\sum_{k=1}^d \partial_{x^k} B^k(Y^x_u)du\right)\right]={\int_{\R^d}f(y)\Gamma(x,y;t)dy}.
\end{align}
{  where $dY^x_t={-B(Y_t^x)} dt + dW_t,~~Y^x_0=x.$}

\end{prop}

 \begin{proof}

For the existence of $\Gamma$, we refer to Theorem 3.5 in Garroni-Menaldi \cite{garroni:menaldi:1992} page 186. See also Theorems 1 and 2 in Krylov \cite{krylov:book} page 68.
%
\\
(i) The fact that ${\cal L}^*$ generates $(Q_t)_{t\ge 0}$ is a consequence of It\^o's formula, more specifically, the Feynman-Kac representation.
\\
The operator ${\cal L}^*$ is densely defined since $C^2_b$ is dense in $L^2$. To show that it is a closed operator, let $f_n \rightarrow 0$ in $L^2$, assume also that ${\cal L}f_n$ converges in $L^2$, then, we have $\lim_{n \rightarrow +\infty}{\cal L}f_n= 0$ in $L^2$, by dominated convergence theorem, since the coefficients are bounded. See also Theorem 1 and 2 in chapter 7 in Evans \cite{evans}. \\
(ii) 
The semi-group property can be deduced from the definition of $(Q_t)_{t \ge 0}$ above and the Markov property { of the process $Y$}:
$$
Y_t^{Y_s^x} \overset{(d)}{=} Y_{t+s}^x.
$$
Indeed, we write:
\begin{eqnarray*}
&&Q_t\circ Q_s(f)(x)\\
&=&E\left[Q_s(f)(Y^x_t)\exp\left(-\int_0^t\sum_{k=1}^d \partial_{x^k} B^k(Y^x_u)du\right)\right]\\
&=&E\left[E\left[f(Y^{Y_t^x}_s)\exp\left(-\int_0^s\sum_{k=1}^d \partial_{x^k} B^k(Y^{Y_t^x}_u)du\right)\right] \exp\left(-\int_0^t\sum_{k=1}^d \partial_{x^k} B^k(Y^x_u)du\right)\right]\\
&=&E\left[f(Y^{x}_{t+s})\exp\left(-\underbrace{\int_0^s\sum_{k=1}^d \partial_{x^k} B^k(Y^{x}_{t+u})du}_{=\int_t^{s+t}\sum_{k=1}^d \partial_{x^k} B^k(Y^{x}_{u})du}\right) \exp\left(-\int_0^t\sum_{k=1}^d \partial_{x^k} B^k(Y^x_u)du\right)\right]\\
&=&E\left[f(Y^x_{t+s})\exp\left(-\int_0^{t+s}\sum_{k=1}^d \partial_{x^k} B^k(Y^x_u)du\right)\right]=Q_{t+s}(f)(x).
\end{eqnarray*}
\\
 (iii)  We now deal with the boundness of $Q_t$ in $L^2$.
\\
Because ${\rm div}B$ is bounded, we clearly have 
$
|Q_t(f)(x)|\le C E\left[|f(Y^x_t)|\right] .
$
\\
From there, we use Girsanov's Theorem to reduce the SDE to the driftless case.
Let $$Z_t:= \exp  \left[ \int_0^t \sum_{k=1}^dB^k(Y_s^x) dW^k_s- \frac{1}{2}\int_0^t \|B(Y_s^x)\|^2 ds\right]$$
and we change measure to ${\mathbb Q}$, the probability measure defined as:
$$\frac{d{\mathbb Q}}{d{\mathbb P}}_{|{\mathcal F}_t} =Z_t.$$ 

To distinguish the expectation under $\mathbb{Q}$ from the expectation under $\mathbb{P}$, we use the notation 
$E_{\mathbb{Q}}$ and $E_{\mathbb{P}}$. Recall $B$ is bounded, we get :
\begin{align*}
|Q_t(f)(x)|&\le C E_{{\mathbb Q}}\left[|f(Y^x_t)|Z_t^{-1}\right] .
\end{align*}
Note that
$$Z^{-1}_t = \exp \left[ - \sum_{k=1}^d \int_0^t B(Y_s^x) dY^x_s -\frac{1}{2} \int_0^t \|B(Y_s^x)\|^2ds \right].$$
According to the Girsanov Theorem, $Y^x-x$ is a ${\mathbb Q}$ Brownian motion issued from 0 and using Cauchy-Schwartz inequality, we get
\begin{align*}
&|Q_t(f)(x)|\\
&\le C E_{{\mathbb P}}\left[|f(x+W_t)|\exp \left[ -\int_0^t\sum_{k=1}^d B^k(x+W_s) dW_s ^k- \frac{1}{2} \int_0^t \|B(x+W_s)\|^2 ds \right]\right]\\
& \le { C}\sqrt{E_{{\mathbb P}}\left[|f(x+W_t)|^2\right]E_{{\mathbb P}}\left[\exp \left[ -2\int_0^t\sum_{k=1}^d B^k(x+W_s) dW_s ^k- \int_0^t \|B(x+W_s)\|^2 ds \right]\right]}.
\end{align*}
Now, to recover a martingale, we write:
\begin{eqnarray*}
&&\exp \left[ -2\int_0^t\sum_{k=1}^d B^k(x+W_s) dW_s ^k- \int_0^t \|B(x+W_s)\|^2 ds \right]\\
&=&\exp \left[ -2 \int_0^t\sum_{k=1}^d B^k(x+W_s) dW_s ^k-2 \int_0^t \|B(x+W_s)\|^2 ds \right]\times \exp \left[ \int_0^t \|B(x+W_s)\|^2 ds \right].
\end{eqnarray*}
The first exponential is a martingale and has expectation 1 under $\PP$, so that:
\begin{eqnarray*}
|Q_t(f)(x)|&\le& { C}\sqrt{E_{{\mathbb P}}\left[|f(x+W_t)|^2\right]E_{{\mathbb P}}\left(\exp \int_0^t \|B(x+W_s)\|^2 ds \right)} \\
&\le& { C e^{\frac12\|B\|_{\infty}^2 T}} \sqrt{E_{{\mathbb P}}\left[|f(x+W_t)|^2\right]}.
\end{eqnarray*}
%
Thus, { using Tonelli Theorem }
we have that $(Q_t)_{t \ge 0}$ is a bounded operator in $L^2$:
\begin{eqnarray}
\label{maj}
\Vert Q_t(f) \Vert_{L^2}^2 &\le&  C^2 e^{\|B\|_{\infty}^2 T}\int_{\R^d} E_{{\mathbb P}}\left[|f(x+W_t)^2|\right] dx \nonumber\\
&=&C^2e^{\|B\|_{\infty}^2 T}E_{{\mathbb P}}\left[   \int_{\R^d}|f(x+W_t)|^2 dx \right] = C^2e^{\|B\|_{\infty}^2 T}\Vert f \Vert_{L^2}^2.
\end{eqnarray}
(iv) 
 We now prove
the strong continuity of the semi-group $Q_t.$
\\
Firstly, for $f$ be a continuous function with compact support,  we expand:
\begin{eqnarray*}
| Q_t(f)(x) - f(x) | &=& 
\left|E_{{\mathbb P}}\Big(f(Y_t^x) \exp\left(-\int_0^t\sum_{k=1}^d \partial_{x^k} B^k(Y^x_u)du\right)- f(x) \Big)\right|\\
&\le& \left|E_{{\mathbb P}}\left[\Big(f(Y_t^x)-f(x)\Big) \exp\left(-\int_0^t\sum_{k=1}^d \partial_{x^k} B^k(Y^x_u)du\right)\right]\right|\\
&&+ \left|E_{{\mathbb P}}\left[f(x) \left(\exp\left(-\int_0^t\sum_{k=1}^d \partial_{x^k} B^k(Y^x_u)du\right)-1 \right)\right]\right|\\
&=& I(t,x) + I\!\!I(t,x).
\end{eqnarray*}
Secondly since ${\rm div B}$ is bounded, we can use the dominated convergence theorem: 
\begin{itemize}
\item Point-wise convergence:
$$
f(x) \left(\exp\left(-\int_0^t\sum_{k=1}^d \partial_{x^k} B^k(Y^x_u)du\right)-1 \right) \underset{t \rightarrow 0}{\longrightarrow} 0.
$$
\item Domination:
$$
\left| f(x) \left(\exp\left(-\int_0^t\sum_{k=1}^d \partial_{x^k} B^k(Y^x_u)du\right)-1 \right) \right| \le C |f(x)|, \quad dx\otimes d\mathbb{P}\mbox{ integrable}
$$
\end{itemize}
Thus, we have
$$
 \int  \Big|I\!\!I(t,x) \Big|^2 dx \underset{t \rightarrow 0}{\longrightarrow} 0.
$$
We now deal with $I(t,x) $.
Since ${\rm div B}$ is bounded, we proceed as in (iii) to get:
$$
I(t,x) = \left|E_{{\mathbb P}}\left[\Big(f(Y_t^x)-f(x)\Big) \exp\left(-\int_0^t\sum_{k=1}^d \partial_{x^k} B^k(Y^x_u)du\right)\right]\right| \le C_T E_{\PP} \Big( f(x+W_t) - f(x)\Big),
$$
and thus conclude using the strong continuity of the Brownian semigroup (note that is is enough to obtain continuity in a dense subspace, due to the domination condition \eqref{maj}). 
%

%
(v) Finally  let $f \in L^2({\mathbb R})$;
 $\forall \varepsilon >0 $
 there exists $g \in L^2\cap C({\mathbb R}^d, {\mathbb R})$ such that
$\|g-f\|_{L^2} <\varepsilon$ and $t_0$ such that for $0 \leq t <t_0,$
$ \|Q_t(g) - g\|_{L^2}\leq \varepsilon.$ Then for $0 \leq t <t_0,$ and using \eqref{maj}
\begin{align*}
\|Q_t(f) -f\|_{L^2} &\leq \| Q_t(f-g) \|_{L^2} + \|f-g\|_{L^2} + \|Q_t(g) - g\|_{L^2}
\\
&\leq 2\varepsilon+ Ce^{4\|B\|_{\infty}^2 T/2}\Vert f-g \Vert_{L^2}.
\end{align*}
Thus, $(Q_t)_{t \ge 0}$ is strongly continuous.

\end{proof}

The following lemmas are a collection of tools needed  in Section \ref{unique}.
\\

 \begin{lemma}
 \label{convolPcen}
 Let $g_\alpha(t):= t^{\alpha/2 -1}.$
 Then for all $n\geq 0,$ $g_\alpha^{*n}(t)=t^{n\alpha/2-1} \frac{(\Gamma(\alpha/2))^n}{\Gamma(n \alpha/2)}$.
 \end{lemma}
  \begin{proof} 
 For $n=1$, this is trivially true. 
 Assume the identity holds for $n \ge 1$. We have:
 $$
 g_\alpha * g_\alpha^{*n}(t) = \int_0^t u^{\alpha/2 -1} (t-u)^{n\alpha/2-1} \frac{(\Gamma(\alpha/2))^n}{\Gamma(n \alpha/2)} du .
 $$
 Changing variables to. $u=tv$, we get:
 $$
 g_\alpha * g_\alpha^{*n}(t) = \int_0^1 v^{\alpha/2 -1} t^{\alpha /2 -1} (1-v)^{n\alpha/2-1}t^{n\alpha/2-1}\frac{(\Gamma(\alpha/2))^n}{\Gamma(n \alpha/2)} tdv 
 = t^{(n+1)\alpha/2-1} B(\alpha/2,n\alpha/2) \frac{(\Gamma(\alpha/2))^{n}}{\Gamma(n\alpha/2)}.
 $$
 Recalling $B(z_1,z_2) = \frac{\Gamma(z_1) \Gamma(z_2) }{ \Gamma(z_1 + z_2)}$ we get the announced identity.
 \end{proof}
  
\begin{lemma}
\label{lemme:approx:dirac}
Let 
$$
\chi(u):= \begin{cases}
 ~a\exp\left( -\frac{1}{1-u^2}\right),&\mbox{if}~~u\in ]-1,1[,\\
~~0 &\mbox{else}
\end{cases}
$$ 
and 
$\chi_\varepsilon(m)=\frac{1}{\varepsilon} \chi(\frac{m-x^1}{\varepsilon}),$ 
 $a$ such that $\int\chi(u)du=1$.  Then when $\varepsilon\rightarrow 0$
there is a  convergence to Dirac measure:
$\lim_{\varepsilon\rightarrow 0}\chi_\varepsilon(m)dx^1=\delta_m(dx^1)
$
meaning for all  $f$ continuous with respect to $x^1$ {and satisfying  $\sup_{x^1} |f(x^1, \tilde x;t)|\in L^1(\R^{d} \times [0,T] )$,}
 $\forall m:$
$$\lim_{\varepsilon\rightarrow 0}\int_0^T\int_{\R^{d}} f(x^1, \tilde x;t){ \chi_\varepsilon(m)}
dx^1d \tilde x dt\rightarrow\int_0^T \int_{\R^{d-1}}f(m, \tilde x;t)d x dt.
$$
\end{lemma}

 \begin{proof}
 
We change of variable to $u=\frac{m-x^1}{\varepsilon}$:
$$\int_0^T\int_{\R^{d}}[f(x^1,\tilde x;t)-f(m,\tilde x;t)]\chi_\varepsilon(x^1)dx^1d\tilde xdt=$$
$$
\int_0^T\int_{\R^{d}}[f(m-u\varepsilon,\tilde x;t)-f(m,\tilde x;t)]\chi(u)dud\tilde xdt.
$$

Since  $f$ is continuous with respect to $x^1$  and
$\sup_{x^1} |f(x^1,\tilde x;t)|\in L^1(\R^{d-1} \times [0,T] )$ the dominated 
Lebesgue Theorem is applied  so yields the expected limit.
\end{proof}

\begin{lemma}
\label{lem-cont-gamma}
Let  $G\in L^2(\R)$,
then $dtd\tilde{x}$ almost surely  $$x^1 \mapsto \int_0^t \int_{\R^d} G(y^1) q(y^1,y;s) \Gamma(x^1,\tilde{x},y;t-s) d\tilde{y}ds$$ is continuous. 
\end{lemma}
 
\begin{proof} { In this proof, $C$ is a generic positive constant whose value may change from line to line and may depend on $T$. For an easier reading, we here recall
 the previous \eqref{rep-gamma-gar},
  \eqref{maj-garro-gamma-1},
  \eqref{maj-garro-der-gamma-1}:
\begin{align*}
\Gamma(x,y;t) =\Gamma_0(x,y;t) + \Gamma_1(x,y;t)
\mbox{ where }
\Gamma_0(x,y;t) =\frac{e^{-\frac{\|x-y\|^2}{2t}}}{\sqrt{2\pi t}^d},
\end{align*}
and for all $y \in R^d, ~~t>0$ $x\mapsto \Gamma_1(x,y;t) \in C^1(\R^d). $ 
Moreover, for $\alpha \in ]0,1[$  there exists some positive constants $C$ and $c$ such that
for all $(x,y,t)\in {\mathbb R}^{2d} \times [0,T],$
cf. \cite{garroni:menaldi:1992}Lemma 3.3 p.184
\begin{align*}
\left| \Gamma_1(x,y;t)\right| \leq Ct^{-\frac{d}{2} +\frac{\alpha}{2}} e^{-\frac{\|x-y\|^2}{ct}}~;~
\left| \partial_{x^l}\Gamma_1(x,y;t)\right| \leq Ct^{-\frac{d+1}{2}+\frac{\alpha}{2}} e^{-\frac{\|x-y\|^2}{ct}}.
\end{align*}
}
 Firstly note that $(\tilde{x},y,t-s),$ $x^1 \mapsto \Gamma(x^1,\tilde{x},y;t-s)$ is continuous (see e.g. Theorem 3.5 page 186 in \cite{garroni:menaldi:1992}). 
 \\
  Secondly we will prove  that 
 \begin{align}
 \label{eq-appendix-lem-ga-1}
 \int_0^T \int_0^t \int_{\R^{2d-1}} |G(y^1) q(y^1,y;s)|\sup_{x^1}| \Gamma(x^1,\tilde{x},y;t-s)| d\tilde{x} dydsdt < \infty.
 \end{align}
 According to estimation {\eqref{maj-garro-gamma-1}},
 and the definition of $\Gamma_0$:
$$
 \sup_{x^1\in {\mathbb R}}\left|\Gamma(x^1,\tilde{x},y^1, \tilde{y};t-s)\right| \leq C(t-s)^{-\frac{d}{2}} e^{-\frac{\|\tilde{x}-\tilde{y}\|^2}{c(t-s)}},~~
 \forall(\tilde{x},y,t) \in \R^{2d-1}\times[0,T].
$$
Integrating with respect to $\tilde{x}$ we obtain
 \begin{align*}
  \int_0^T \int_0^t \int_{\R^{2d-1}} |G(y^1) q(y^1,y;s)|\sup_{x^1}| \Gamma(x^1,\tilde{x},y;t-s)| d\tilde{x} dydsdt\\
  \leq Cc^{d-1}
 \int_0^T \int_0^t \int_{\R^{d}} |G(y^1) q(y^1,y;s)|
 {(t-s)^{-1/2}} dydsdt .
 \end{align*} 
 Using the fact that $G\in L^2(\R)$ and Cauchy-Schwartz inequality,
 \begin{align*}
   \int_0^T \int_0^t \int_{\R^{2d-1}} |G(y^1) q(y^1,y;s)|\sup_{x^1}| \Gamma(x^1,\tilde{x},y;t-s)| d\tilde{x} dydsdt\\
\leq Cc^{d-1} \|G\|_{L^2}\int_0^T \int_0^t \sqrt{\int_{R^d} { q^2(y^1,y;s)}dy} {(t-s)^{-1/2}}dsdt
 \end{align*} 
Using Tonelli's theorem, we obtain
\begin{align*}
   \int_0^T \int_0^t \int_{\R^{2d-1}} |G(y^1) q(y^1,y;s)|\sup_{x^1}| \Gamma(x^1,\tilde{x},y;t-s)| d\tilde{x} dydsdt\\
\leq Cc^{d-1} \|G\|_{L^2}\int_0^T \sqrt{\int_{R^d} 
{ q^2(y^1,y,s)}dy}\int_s^T {(t-s)^{-1/2}} dt ds.
 \end{align*} 
 Then,
 \begin{align}
 \label{appendix-maj-0}
   \int_0^T \int_0^t \int_{\R^{2d-1}} |G(y^1) q(y^1,y;s)|\sup_{x^1}| \Gamma(x^1,\tilde{x},y;t-s)| d\tilde{x} dydsdt\\
\leq 2T^{1/2}Cc^{d-1} \|G\|_{L^2}\int_0^T \sqrt{\int_{R^d} { q^2(y^1,y,s)}dy}ds<\infty\nonumber
 \end{align}
 according to Item (b) 
 of Definition \ref{def-espace-sol}.
 That means that $d\tilde{x}dt$ almost surely
  \begin{align}
  \label{appendix-def-neg-0}
 \int_0^t \int_{\R^{d}} |G(y^1) q(y^1,y;s)|\sup_{x^1}| \Gamma(x^1,\tilde{x},y;t-s)| dyds<\infty.
 \end{align}
  Thirdly for all $(\tilde{x},t)$ such that \eqref{appendix-def-neg-0} holds,
 according to the dominated convergence theorem 
  $x^1 \mapsto \int_0^t \int_{\R^d} G(y^1) q(y^1,y;s) \Gamma(x^1,\tilde{x},y;t-s) d\tilde{y}ds$ is continuous (see Theorem 3.5 page 186 in \cite{garroni:menaldi:1992}).

\end{proof} 
 
\begin{lemma}
\label{lem-cont-gamma-2}
{ Let $F~:{\mathbb R}^{d-1} \rightarrow {\mathbb R}$ be a continuous function with compact support, then $dtdy^1$ almost surely  $x^1 \mapsto \int_0^t \int_{\R^{2d-2}} F(\tilde{x}) q(y^1,y^1,\tilde{y};s) \Gamma(x^1,\tilde{x},y^1,\tilde{y};t-s) d\tilde{y}dsd\tilde{x}$ is continuous.}
\end{lemma}
 \begin{proof} Let $G$ be the indicator of the intervall $[n,n].$ 
According to estimation \eqref{appendix-maj-0} and since $F$ is bounded,
almost surely $dtdy^{1}$ on $[0,T]\times [-n,n]$
\begin{align}
\label{appendix-def-neg-2}
\int_0^t \int_{\R^{2d-2}} |F(\tilde{x})|| q(y^1,y^1,\tilde{y},s)|\sup_{x^1} \Gamma(x^1,\tilde{x},y^1,\tilde{y};t-s)| d\tilde{y}dsd\tilde{x}<\infty.
\end{align}
Since $(\tilde{x},y,t-s)$ $x^1 \mapsto \Gamma(x^1,\tilde{x},y;t-s)$ is continuous,  using Lebesgue dominated theorem for $(t,y^1)$ such that \eqref{appendix-def-neg-2} holds:
\\
$x^1 \mapsto \int_0^t \int_{\R^{2d-2}} F(\tilde{x}) q(y^1,y^1,\tilde{y},s) \Gamma(x^1,\tilde{x},y^1,\tilde{y};t-s) d\tilde{y}dsd\tilde{x}$ is continuous.
\\
Since $\cup_n [-n,n]=\R$ this achieves the proof of Lemma \ref{lem-cont-gamma-2}.
 
\end{proof} 
\section{Conclusion and perspectives}

In this work, we have have established the uniqueness of the solution to equation \eqref{edp1m}. Together with the findings of \cite{coutin:pontier:2019} and \cite{coutin:pontier:2020}, which address the governing equation for a diffusion process and its running supremum, this trilogy of studies establishes that the joint law is absolutely continuous with a smooth density, that the density admits a series expansion satisfying a non-standard PDE, and that this PDE has a unique solution. In this section, we explore potential research directions stemming from these contributions

First and foremost, a natural extension of our model would allow for a diffusion coefficient. In light of Frikha and Li \cite{Frikha:li:2020}, it should also be feasible to extend our approach to path-dependent cases, or at least to cases where the coefficients may depend on the supremum. Another direction for future investigation is the case where the coefficients have low regularity, a setting for which the parametrix method is particularly well-suited.

Regarding PDEs with boundary conditions, there are numerous ways to address low regularity in the coefficients. One approach might involve modifying the notion of uniqueness, as a pointwise (or almost everywhere) definition often fails in low-regularity contexts, requiring the use of $L^p$ viscosity solutions. For example, the strategy described in \cite{Bahlali:Boufoussi:Mouchtabih:2022} could be relevant here.

Focusing on our method, one could ask what is the most general case in which it guarantees uniqueness. In abstract frameworks such as \cite{Hoffmann:Wald:Nguyen:2022}, specific conditions on the operators ensure uniqueness. Naturally, one might wonder whether our approach fits into their framework or if it could inspire a different method for proving uniqueness. In this regard, the method proposed in Frikha-Li \cite{Frikha:li:2020}, building on the work of Bass and Perkins \cite{bass:perkins}, could also be examined abstractly to explore the smoothing properties of the kernel in a broader setting.

\acks 
\noindent The authors would like to thank both referees and the editor for their insightful remarks, as well as Romain Dubosc, Moritz Kassmann, and Philippe Laurençot for helpful discussions on this problem.
\\


\competing 
\noindent There were no competing interests to declare which arose during the preparation or publication process of this article.

%
%
%
%

%
%
%
%

\end{document}